\documentclass[10pt,a4paper,oneside]{amsart}
\usepackage{amssymb}
\usepackage{mathrsfs}  
\usepackage{graphicx}
\usepackage[svgnames]{xcolor}
\usepackage{epstopdf}
\usepackage[ps2pdf]{hyperref}  
\hypersetup{colorlinks=true, linkcolor=blue, anchorcolor=blue, 
citecolor=red, filecolor=blue, menucolor=blue, pagecolor=blue, 
urlcolor=blue}

\numberwithin{equation}{section}
\newtheorem{theorem}{Theorem}[section]

\newtheorem{lemma}[theorem]{Lemma}
\theoremstyle{remark}
\newtheorem{remark}{Remark}[section]

\theoremstyle{definition}

\newcommand{\xx}{\langle x\rangle}

\newcommand{\mmu}{\langle \mu\rangle}
\newcommand{\one}{\mathbf{1}}
\newcommand{\onep}{\one_{+}}

\newcommand{\sgn}{\mathop{\mathrm{sgn}}}

\begin{document}

\title%
[Steplike Potential]%
{Dispersive estimate for the 1D Schr\"odinger\\ 
equation with a steplike potential}

\begin{abstract}
  We prove a sharp dispersive estimate 
  \begin{equation*}
    |P_{ac}u(t,x)|\le C|t|^{-1/2}\cdot\|u(0)\|_{L^{1}(\mathbb{R})}
  \end{equation*}
  for the one dimensional Schr\"odinger equation
  \begin{equation*}
    iu_{t}-u_{xx}+V(x)u+V_{0}(x)u=0,
  \end{equation*}
  where $(1+x^{2})V\in L^{1}(\mathbb{R})$ and $V_{0}$ is a step
  function, real valued and consant on the positive and negative real
  axes.
\end{abstract}

\date{\today}    

\author{Piero D'Ancona}
\address{Piero D'Ancona: 
SAPIENZA - Universit\`a di Roma,
Dipartimento di Matematica, 
Piazzale A.~Moro 2, I-00185 Roma, Italy}
\email{dancona@mat.uniroma1.it}

\author[Sigmund Selberg]{Sigmund Selberg}
\address{Department of Mathematical Sciences\\
Norwegian University of Science and Technology\\
Alfred Getz' vei 1\\
N-7491 Trondheim\\ Norway}
\email{sigmund.selberg@math.ntnu.no}

\subjclass[2000]{
35J10, 
35Q40, 
58J45
}
\keywords{%
steplike potentials, 
Schr\"{o}dinger equation,
dispersive estimate}
\maketitle

\section{Introduction}\label{sec:intro} 

The \emph{dispersive estimate} for the Schr\"odinger equation
on $\mathbb{R}^{n}$, $n\ge1$,
\begin{equation*}
  iu_{t}-\Delta u=0,\qquad u(0,x)=f(x)
\end{equation*}
states that, for all $f\in L^{1}(\mathbb{R}^{n})$
and $t\neq0$, the solution satisfies
\begin{equation}\label{eq:dispgen}
  |u(t,x)|\le C|t|^{-\frac n2}\|f\|_{L^{1}(\mathbb{R}^{n})}.
\end{equation}
The sharp constant is $C=(4\pi)^{-n/2}$.
Estimate \eqref{eq:dispgen} is elementary and follows from
the explicit form of the fundamental solution; 
nevertheless, it represents the starting point
for a large number of important developments including Strichartz 
estimates and the local and global well posedness theory for 
nonlinear Schr\"odinger equations. Thus the problem of extending
\eqref{eq:dispgen} to more general equations has received a great
deal of attention.

Potential perturbations of the form
\begin{equation*}
  iu_{t}-\Delta u+V(x)u=0
\end{equation*}
(on $\mathbb{R}^{n}$, $n\ge3$) were considered in many papers,
starting with
\cite{JourneSofferSogge91-a}, with improvements at several
reprises (see e.g.~\cite{Yajima95-b}, \cite{Yajima95-a}, 
\cite{Yajima99-a},
\cite{GoldbergSchlag04-b},
\cite{DanconaPierfeliceTeta06-a}).
Focusing on the one dimensional case
\begin{equation}\label{eq:pot1D}
  iu_{t}-u_{xx}+V(x)u=0,
\end{equation}
which is the subject of this paper,
the first proof of the dispersive estimate
is surprisingly recent and due to Weder \cite{Weder00-b}. His result was
improved by Artbazar and Yajima \cite{ArtbazarYajima00}, who
actually proved the more general fact that the wave operator associated to
$-\frac{d^{2}}{dx^{2}}+V(x)$ is bounded on $L^{p}$ for all $p$.
Finally, Goldberg and Schlag \cite{GoldbergSchlag04-b}
proved \eqref{eq:dispgen}
for potentials satisfying $(1+x^{2}) V\in L^{1}(\mathbb{R})$,
or the weaker condition $(1+|x|) V\in L^{1}(\mathbb{R})$ plus an
additional nonresonant condition at $0$; these conditions on the
potential are conjectured to be optimal. Under the same assumptions
on $V$, the $L^{p}$ boundedness of the wave operator
was proved by D'Ancona and Fanelli \cite{DanconaFanelli06-a}.
We also mention that potentials with slower decay
present new phenomena as evidenced in \cite{ChristKiselev01-c},
\cite{ChristKiselev02-a}.

Thus the problem of dispersive estimates for
potential perturbations is essentially settled in 1D if the
potential is small in the appropriate sense
at infinity. Notice that we can add a real constant to $V$
without modifying the dispersive properties via the gauge
transformation $u\to u e^{i \lambda t}$. Hence a more precise
statement is that dispersion has been proved for 
potentials having the same asymptotic behaviour at $x\to\pm \infty$.

Here we consider a more general
kind of potential with possibly different asymptotic behaviours
at $+ \infty$ and $- \infty$.
We call a \emph{steplike potential} a potential of the form
\begin{equation*}
  V(x)+V_{0}(x)  
\end{equation*}
where
\begin{equation*}
  V\in L^{1}(\mathbb{R}), \qquad
  V_{0}(x)
  =
  \left\{ 
  \begin{array}{l@{\quad:\quad}l}
  V_{-} & x < 0 \\
  V_{+} & x > 0,
  \end{array}\right.
  \qquad
  V_{-},V_{+}\in \mathbb{R}.
\end{equation*}
It is not restrictive to assume $V_{-}<V_{+}$ as we shall do
from now on. 

In the physical literature, steplike potentials are also called
\emph{barrier potentials} and are used 
to model the interaction of particles
with the boundary of solids (see 
\cite{GesztesyNowellPotz97-a}
for a general
discussion of problems with nontrivial asymptotics).

Steplike potentials occur also in general relativity. 
An example is given by the radial Klein-Gordon equation
$\square_{g}u-m^{2}u=0$ on a (radial) Schwarzschild background
\begin{equation*}
  g=-\left(1-\frac{2M}{r}\right)dt^{2}+
    \left(1-\frac{2M}{r}\right)^{-1}dr^{2}
\end{equation*}
Here $M>0$ is related to the mass of the black hole and
$r>2M$ is a radial variable.
If we 
introduce the Regge-Wheeler coordinate
\begin{equation}\label{eq:RW}
  s=r+2M\log\left(\frac{r-2M}{2M}\right)
\end{equation}
and denote with $r(s)$ the inverse
function of \eqref{eq:RW}, the radial Klein-Gordon 
equation takes the form
\begin{equation*}
  u_{tt}-u_{ss}+v(s)u=0,\qquad
  v(s)=
  \left(m^{2}+\frac{2M}{r(s)^{3}}\right)
  \left(1-\frac{2M}{r(s)}\right).
\end{equation*}
It is easy to check that the asymptotic behaviour of $v(s)$ is
\begin{equation*}
  v(s)\sim
  \begin{cases}
    m^{2}+O(s^{-3}) &\text{for $ s\to+ \infty $}\\
    O(e^{(2M)^{-1}s}) &\text{for $ s\to-\infty $,}
  \end{cases}
\end{equation*}
so that the equation is reduced to
a wave equation perturbed with a steplike potential.

Despite the significance of this class of potentials, mathematical studies
have considered only the problem of direct and
inverse scattering, with the usual applications to the 
Korteweg-de Vries equation, following the
classical theory of \cite{Faddeyev63-a} and \cite{DeiftTrubowitz79-a}.
A quite detailed theory was established in
\cite{BuslaevFomin62-a},
\cite{CohenKappeler85-a},
\cite{DaviesSimon78-a},
\cite{Gesztesy86-a},
\cite{Aktosun99-a},
\cite{Christiansen06-a},
\cite{Cohen84-a},
\cite{Kappeler86-a}.
See also \cite{Boutet-de-MonvelEgorovaTeschl08-a} for a more
recent take on this class of problems.
.

On the other hand, to our knowledge, the dispersive properties
of evolution equations perturbed with steplike potentials have never
been investigated. Our goal here is to initiate this subject.
The basic model is the Schr\"odinger equation
\begin{equation}\label{eq:schrstep}
  iu_{t}-u_{xx} + V_{0}(x) u=0,\qquad
  u(0,x)=f(x)
\end{equation}
where as above $V_{0}(x)$ is a piecewise constant function equal to
$V_{+}$ for $x>0$ and to $V_{-}$ for $x<0$, $V_{-}<V_{+}$.
As a preparation to the study of more general 
potentials, in Section \ref{sec:proof1} we compute 
explicit kernels both for the resolvent
of the operator $-d^{2}/dx^{2}+V_{0}$ and for the fundamental solution
to \eqref{eq:schrstep}. As a first application, we
prove that any solution of \eqref{eq:schrstep} satisfies the
dispersive estimate
\begin{equation}\label{eq:dispstep}
  |u(t,x)|\le C\|f\|_{L^{1}} \cdot |t|^{-1/2},\qquad
  t\neq0.
\end{equation}
It is easy to check by a gauge transform and a rescaling
$u\to e^{i \lambda t}u(\alpha^{2}t,\alpha x)$ that the constant
in the dispersive estimate does not depend on $V_{\pm}$.

After this preliminary study of the model case \eqref{eq:dispstep},
we pass to the general situation of a Schr\"odinger equation
\begin{equation}\label{eq:schrstepgen}
  iu_{t}- u_{xx} +V_{0}u+V(x)u=0,\qquad
  u(0,x)=f(x)
\end{equation}
where $V_{0}$ is perturbed with a potential
$V(x)$ belonging to a suitable weighted $L^{1}$ class. 
In order to state our result, we recall that the operator
$-d^{2}/dx^{2}+V+V_{0}$ typically has a nonempty point spectrum,
contained in $(-\infty,0]$; since bound states do not disperse
we need to project them away.
We denote by $P_{ac}$ the projection on the absolutely continuous
subspace of $L^{2}(\mathbb{R})$ associated to the operator.
Then the main result of the paper is the following:

\begin{theorem}\label{the:steppot}
  For any $f\in L^{1}(\mathbb{R})$ and any real valued potential
  $V(x)$ satisfying 
  \begin{equation}\label{eq:stepVgen}
    (1+x^{2})V(x)\in L^{1}(\mathbb{R}),
  \end{equation}
  the solution $u(t,x)$
  to \eqref{eq:schrstepgen} satisfies the dispersive estimate
  \begin{equation}\label{eq:dispstepgen}
    |P_{ac}u(t,x)|\le C\|f\|_{L^{1}} \cdot |t|^{-1/2}.
  \end{equation}
\end{theorem}

Section \ref{sec:general_step_pot} is devoted to the proof of
Theorem \ref{the:steppot}. As usual, we need to treat high 
frequencies
and low frequencies with separate methods. Notice that for
the high frequency part of the solution, dispersion can be
proved under the weaker assumption $V\in L^{1}(\mathbb{R})$.
The low frequency part is more difficult to estimate, and
requires some rather precise information on the asymptotic
behaviour of Jost solutions for the corresponding Helmholtz
equation.

\begin{remark}\label{rem:optimal}
  It is natural to question the optimality of the assumption
  $(1+x^{2})V\in L^{1}$. A reasonable conjecture is that
  $(1+x^{2})^{\gamma/2}V(x)\in L^{1}$ with $\gamma\ge1$
  should be enough, provided some spectral assumption is
  made to exclude resonance at 0. In the classical case
  $V_{0}\equiv 0$, the standard assumption is
  that the equation
  \begin{equation*}
    -f''+V(x)f=0
  \end{equation*}
  has two linearly independent solutions $f_{+},f_{-}$ 
  with the asymptotic behaviour
  \begin{equation*}
    f_{+}\sim 1\ \text{as $x\to+\infty$,}\qquad
    f_{-}\sim 1 \ \text{as $x\to-\infty$.}
  \end{equation*}
  Potentials $V(x)$ satisfying this condition are called
  \emph{generic}, while they are called \emph{exceptional} when the
  condition fails. Notice that $V \equiv 0$ is of exceptional type.
  When $V_{0}$ is not zero, e.g., the Heaviside function, the
  asymptotic for $f_{+}$ should be modified to require
  $f_{+}\sim e^{-x}$ as $x\to+\infty$.
  We prefer not to pursue here the delicate question of the optimal
  assumptions on the potential.
\end{remark}

The main application of dispersive estimates concerns
problems of global well posedness and stability for
nonlinear perturbations of the equation. We shall devote further
works to this aspect, while here we shall focus on the linear
estimates exclusively.


\section{Resolvent operator and dispersive estimate for
$-d^{2}/dx^{2}+V_{0}$}
\label{sec:proof1}  

Throughout the paper, for complex numbers $z$
we denote by
\begin{equation*}
  z^{1/2}
  \ \text{the square root of $z$ with $\Im z^{1/2}\ge0$.}
\end{equation*}
For any $z\not\in \mathbb{R}$,
let $R_{0}(z)=(H_{0}-z)^{-1}$  be the resolvent of the
selfadjoint operator $H_{0}=-\frac{d^{2}}{dx^{2}}+V_{0}$, where
$V_{0}$ is the piecewise constant function
\begin{equation*}
  V_{0}=V_{-}\quad\text{for $x<0$},\qquad
  V_{0}=V_{+}\quad\text{for $x>0$}.
\end{equation*}
Both the selfadjointness of $H_{0}$ and the fact that
$\sigma(H_{0})=[V_{-},+\infty)$ follow from the standard theory.
Denoting by $r_{\pm}(z)$ the functions
\begin{equation}\label{eq:rpm}
  r_{\pm}=(z-V_{\pm})^{1/2}, \qquad
  z\not\in [V_{-},+\infty)
\end{equation}
we see that $R_{0}$ can be represented as an integral operator
\begin{equation*}
  R_0(z)f=\int K^{0}_{z}(x,y)f(y)dy
\end{equation*}
where the kernel $K^{0}_{z}(x,y)$ is expressed by the following
formulas: for all $y<x\in \mathbb{R}$,
\begin{equation}\label{eq:Kz}
  K^{0}_{z}(x,y)=K^{0}_{z}(y,x)=
  \begin{cases}
    \frac{1}{2ir_{-}}e^{ir_{-}(x-y)}+
    \frac{1}{2ir_{-}}
    \frac{r_{-}-r_{+}}{r_{-}+r_{+}}e^{ir_{-}(-x-y)}
     &\text{if $ y<x<0 $,}\\
    \frac{1}{i(r_{+}+r_{-})}e^{ir_{+}x}e^{-ir_{-}y}
     &\text{if $ y<0<x $,}\\
   \frac{1}{2ir_{+}}
   \frac{r_{-}-r_{+}}{r_{-}+r_{+}}e^{ir_{+}(x+y)}+
   \frac{1}{2ir_{+}}e^{ir_{+}(x-y)}
     &\text{if $ 0<y<x $.}
  \end{cases}
\end{equation}
The explicit formula for $K^{0}_{z}$ can be computed
by an elementary application of the standard
theory of ordinary differential equations. Notice that the kernel has
a well defined limit as $z$ approaches a point of the spectrum of
$H_{0}$ from above and from below (with different limits).

The spectral formula allows to represent any function $\phi(H_{0})$
of the operator, for sufficiently nice $\phi$, as the 
$L^{2}$ limit
\begin{equation}\label{eq:spec}
  \phi(H_{0})f=\lim_{\epsilon \downarrow0}
  \int_{\gamma_{\epsilon}} \phi(z)R_0(z)fdz
\end{equation}
over a curve $\gamma_{\epsilon}$ which in the present case can 
be taken as
the union of the straight half lines $z=\lambda\pm i \epsilon$,
$\lambda>V_{-}$, with the left semicircle of radius
$\epsilon$ around $z=V_{-}$, from $+\infty-i \epsilon$ to
$+\infty+i \epsilon$. After the change of variables 
$z= w^{2}+V_{-}$ we can rephrase the spectral formula as
\begin{equation}\label{eq:spec1}
  \phi(H_{0})f=\int_{-\infty}^{+\infty}
  \phi(\lambda+V_{-})\lambda R_0(\lambda^{2}+V_{-})f d \lambda
\end{equation}
where the limit operator
\begin{equation*}
  R_0(\lambda^{2}+V_{-})f=\int K_{\lambda}(x,y)f(y)dy
\end{equation*}
has a kernel $K_{\lambda}$ given by the following expressions
valid for all $y<x\in \mathbb{R}$:
\begin{equation}\label{eq:Kl}
  K_{\lambda}(x,y)=K_{\lambda}(y,x)=
  \begin{cases}
    \frac{1}{2i \lambda}e^{i \lambda(x-y)}+
    \frac{1}{2i\lambda}
    \frac{\lambda-\rho_{+}}{\lambda+\rho_{+}}e^{i\lambda(-x-y)}
     &\text{if $ y<x<0 $,}\\
    \frac{1}{i(\rho_{+}+\lambda)}e^{i\rho_{+}x}e^{-i\lambda y}
     &\text{if $ y<0<x $,}\\
   \frac{1}{2i\rho_{+}}
   \frac{\lambda-\rho_{+}}{\lambda+\rho_{+}}e^{i\rho_{+}(x+y)}+
   \frac{1}{2i\rho_{+}}e^{i\rho_{+}(x-y)}
     &\text{if $ 0<y<x $}
  \end{cases}
\end{equation}
and the function $\rho_{+}(\lambda)$ is defined as
\begin{equation}\label{eq:rho}
  \rho_{+}(\lambda)=
  \begin{cases}
    (\lambda^{2}-\delta^{2})^{1/2}
     &\text{if $ \lambda>\delta $,}\\
    i(\delta^{2}-\lambda^{2})^{1/2}
     &\text{if $ -\delta<\lambda<\delta $,}\\
    -(\lambda^{2}-\delta^{2})^{1/2}
     &\text{if $ \lambda<-\delta $.}
  \end{cases}\qquad
  \delta^{2}=V_{+}-V_{-}\ge0.
\end{equation}
Notice that \eqref{eq:Kl} is a simple instance
of the limiting absorption principle; indeed, the kernel $K_{\lambda}$
is a bounded function and hence defines a bounded operator 
from $L^{2}(\xx^{1+}dx)$ to $L^{2}(\xx^{-1-}dx)$ as in the standard
theory.

We can now use \eqref{eq:spec1}, \eqref{eq:Kl}, \eqref{eq:rho}
to represent the solution of the Schr\"odinger equation 
\eqref{eq:schrstep} via the formula
\begin{equation}\label{eq:expl}
  e^{itH_{0}}f=
  e^{itV_{-}}
  \int_{-\infty}^{+\infty}\left(
  \int_{\mathbb{R}}
  e^{it \lambda^{2}}\lambda
  K_{\lambda}(x,y)f(y)dy
  \right)d \lambda.
\end{equation}
The factor $e^{itV_{-}}$ is inessential and will be
dropped from now on (i.e.~we can assume $V_{-}=0$). 
Moreover, as remarked in the Introduction, by a rescaling we
can assume $\delta=1$. Thus from now on the function $V_{0}$
will be taken equal to the Heaviside function
\begin{equation*}
  V_{0}(x)=\one_{+}(x)=\one_{[0,+\infty)}(x).
\end{equation*}

In order to
swap the integrals and perform some stationary phase calculations we
introduce a cutoff function with the following properties:
$\chi\in C^{\infty}_{c}(\mathbb{R})$ is an even function,
\begin{equation}\label{eq:chi}
  \chi(\lambda)=
  \begin{cases}
    0 &\text{if $ |\lambda|\ge8 $,}\\
    1 &\text{if $ |\lambda|\le4 $,}
  \end{cases}
  \qquad
  |\chi'|\le1,
  \qquad
  |\chi''|\le1.
\end{equation}
Then we shall study the operator
\begin{equation}\label{eq:trunc}
  \chi_{M}(H_{0})e^{itH_{0}}f=
  \int_{\mathbb{R}}
  \int_{-\infty}^{+\infty}
  e^{it \lambda^{2}}\lambda \chi_{M}(\lambda)
  K_{\lambda}(x,y)d \lambda\ f(y)
  dy,
  \qquad
  \chi_{M}(\lambda)=\chi\left(\frac{\lambda}{M}\right).
\end{equation}
Notice that, in order to prove \eqref{eq:dispstep},
it will be sufficient to prove a dispersive $L^{1}-L^{\infty}$
estimate for the truncated operator \eqref{eq:trunc}, 
uniform in $M>>1$.

Finally, it is clearly sufficient to prove the pointwise estimate 
\begin{equation}\label{eq:ptw}
  \left|
  \int_{-\infty}^{+\infty}
  e^{it \lambda^{2}}\lambda \chi_{M}(\lambda)
  K_{\lambda}(x,y)d \lambda
  \right|\le
  C|t|^{-1/2}
\end{equation}
for some constant $C$ independent of $M,x$ and $y$.
The rest of this section is devoted to the proof of \eqref{eq:ptw}.
By symmetry
we can assume $y<x$, and recalling \eqref{eq:Kl} we shall handle
separately the three cases $y<x<0$, $y<0<x$ and $0<y<x$.

\subsection{First case: $y<x<0$}\label{sub:first_case}  

Apart from inessential numeric factors, we must estimate the two terms
\begin{equation}\label{eq:Ia}
  v(t,A)=\int  
    \chi_{M}(\lambda)
    e^{it \lambda^{2}}
    e^{i \lambda A}
    d \lambda,\qquad
    A=x-y>0
\end{equation}
and
\begin{equation}\label{eq:Ib}
  w(t,A)=\int 
    \chi_{M}(\lambda)
    \frac{\lambda-\rho_{+}}{\lambda+\rho_{+}}
    e^{it \lambda^{2}}
    e^{i\lambda A}
    d \lambda,\qquad
    A=-x-y>0
\end{equation}
uniformly in $M$ and $y<x<0$. 
By a standard trick
we regard $v(t,A)$ as a solution of the 1D Schr\"odinger equation
\begin{equation*}
  iv_{t}-v_{AA}=0
\end{equation*}
and hence we can apply the usual dispersive estimate
\begin{equation*}
  |v|\le |t|^{-1/2}\|v(0,A)\|_{L^{1}_{A}},\qquad
  \|v(0,A)\|_{L^{1}_{A}}=\|\widehat{\chi_{M}}\|_{L^{1}}=
  \|\widehat{ \chi}\|_{L^{1}}.
\end{equation*}
Thus we get
\begin{equation}\label{eq:Iaest}
  |v|\le \|\widehat{ \chi}\|_{L^{1}} \cdot|t|^{-1/2}
\end{equation}

A similar strategy applied to $w(t,A)$ gives
\begin{equation*}
  |w|\le |t|^{-1/2}\|I\|_{L^{1}_{A}}
\end{equation*}
with
\begin{equation}\label{eq:I}
  I=
  \int
  \frac{\lambda-\rho_{+}}{\lambda+\rho_{+}}
  e^{i \lambda A}
  \chi_{M}(\lambda)
  d \lambda.
\end{equation}
Notice that for $\lambda>1$ (recall that $\delta=1$)
\begin{equation*}
  \frac{\lambda-\rho_{+}}{\lambda+\rho_{+}}=
  \frac{1}{(\lambda+\rho_{+})^{2}}=
  \frac{1}{\left(\lambda+(\lambda-1)^{1/2}\right)^{2}}
\end{equation*}
and analogously for 
$\lambda<-1$, $-1<\lambda<1$, thus we have
\begin{equation*}
  I=
  \int \frac{1}{(\lambda+\rho_{+})^{2}}
  e^{i \lambda A }\chi_{M}d \lambda.
\end{equation*}
Introduce now an additional cutoff $\psi(\lambda)$
in the integral to isolate the H\"older singularity around $\lambda=1$
of $\rho_{+}$, i.e., choose $\psi\in C^{\infty}(\mathbb{R})$ with
\begin{equation*}
  \psi(\lambda)=
  \begin{cases}
    1 &\text{if $ \lambda\ge3 $,}\\
    0 &\text{if $ \lambda\le2 $}
  \end{cases}
\end{equation*}
and consider the corresponding piece of the integral
\begin{equation}\label{eq:Ib01}
  I_{1}=
  \int
  \frac{\psi}{(\lambda+\rho_{+})^{2}}
  e^{i \lambda A}
  \chi_{M}(\lambda)
  d \lambda.
\end{equation}
Two integrations by parts give
\begin{equation*}
  \|I_{1}\|_{L^{1}_{A}}\le
  \left\|
    \frac{\psi \chi_{M}}{(\lambda+\rho_{+})^{2}}
  \right\|_{W^{2,1}_{\lambda}}
\end{equation*}
which is uniformly bounded for $M\ge1$.

The piece of $I$ near $\lambda=1$ is more delicate.
We choose a new cutoff which we denote again by
$\psi\in C^{\infty}_{c}$ supported in
$[-1/2,3]$ and consider $I_{2}$ given by
\begin{equation}\label{eq:Ib02}
  I_{2}=
  \int
  \frac{\lambda-\rho_{+}}{\lambda+\rho_{+}}
  e^{i \lambda A}
  \psi
  \chi_{M}
  d \lambda=
  \int
  \frac{\lambda-\rho_{+}}{\lambda+\rho_{+}}
  e^{i \lambda A}
  \psi
  d \lambda
\end{equation}
where we can suppress $\chi_{M}$ since $\psi \chi_{M}\psi$ for $M>1$. 
We decompose the singular factor into even and odd part using the identity
\begin{equation*}
  \frac{\lambda-s}{\lambda+s}=
  \frac{\lambda^{2}+s^{2}}{\lambda^{2}-s^{2}}-
  \frac{2s \lambda}{\lambda^{2}-s^{2}}
\end{equation*}
obtaining (recall $\rho_{+}^{2}=\lambda^{2}-\delta^{2}=\lambda^{2}-1$)
\begin{equation*}
  \frac{\lambda-\rho_{+}}{\lambda+\rho_{+}}=
  (2 \lambda^{2}-1)- 2 \lambda \rho_{+}.
\end{equation*}
The term in $(2 \lambda^{2}-1)$ is trivial since
(apart from numeric factors)
\begin{equation*}
  \int(2 \lambda^{2}-1)\psi e^{i \lambda A}d \lambda=
  -2\widehat{\psi }''(A)-\widehat{\psi }(A)
\end{equation*}
which is certainly $L^{1}$ bounded. On the other hand, we can split
\begin{equation*}
  \int \lambda\psi \rho_{+}e^{i \lambda A}d \lambda=
  \int_{1}^{+\infty}+\int_{-\infty}^{1}
\end{equation*}
and we shall focus on the first term since the second one is 
entirely analogous. We rewrite it as
\begin{equation*}
  \int_{1}^{+\infty}\lambda\psi \rho_{+}e^{i \lambda A}d \lambda=
  \int_{0}^{\infty}\psi_{1}(\lambda)\lambda^{1/2}e^{i \lambda A}d \lambda
     \cdot e^{iA},\qquad
  \psi_{1}(\lambda)=(\lambda+1)^{3/2}\psi(\lambda+1)
\end{equation*}
and notice that, again, $\psi_{1}\in C^{\infty}_{c}$. 
This kind of integral is standard; applying e.g. estimate \eqref{eq:phase2}
in the Appendix we get
\begin{equation*}
  \left|
    \int_{1}^{+\infty}\lambda\psi \rho_{+}e^{i \lambda A}d \lambda
  \right|\le C(\psi) |A|^{-3/2},
\end{equation*}
while we have directly
\begin{equation*}
  \left|
    \int_{1}^{+\infty}\lambda\psi \rho_{+}e^{i \lambda A}d \lambda
  \right|\le C(\psi).
\end{equation*}
Summing up we obtain  $\|I_{2}\|_{L^{1}_{A}}\le C$ as required.
The estimate of the remaining part of $I$ is identical, using
cutoffs supported in $(-\infty,-2]$ and $[-3.1/2]$ as above.


\subsection{Second case: $y<0<x$}\label{sub:second_case}  

The quantity \eqref{eq:ptw} takes the form ($y\to-y$)
\begin{equation*}
  \int \frac{\lambda}{\rho_{+}+\lambda}
  e^{i \rho_{+}x}e^{i \lambda y}e^{i\lambda^{2}t}
  \chi_{M}d \lambda,\qquad
  x,y>0.
\end{equation*}
We split the integral in three parts
\begin{equation*}
  I+II+III=\int_{1}^{+\infty}+\int_{-1}^{1}+\int_{-\infty}^{-1}.
\end{equation*}
After a reflection $\lambda\to-\lambda$ 
and a conjugation in $III$ we see that
parts $I$ and $III$ have the form
\begin{equation*}
  I,III=
  \int_{1}^{\infty} \frac{\lambda}{\rho_{+}+\lambda}
  e^{i \rho_{+}x}e^{i \lambda y}e^{\pm i\lambda^{2}t}
  \chi_{M}d \lambda,\qquad
  x,y>0.
\end{equation*}
The basic tool will be the following Lemma:

\begin{lemma}\label{lem:vdc}
  For all $b>1$, $x,y>0$,
  \begin{equation}\label{eq:vdc}
    \left|
    \int_{1}^{b}
     \frac{\lambda}{\rho_{+}+\lambda}
     e^{i \rho_{+}x}e^{i \lambda y}e^{\pm i\lambda^{2}t}
   \right|\le 15|t|^{-1/2}.
  \end{equation}
\end{lemma}

\begin{proof}
  We recall one form of the van der Corput Lemma
  (see \cite{Stein86-a}:
  for any $C^{2}$ function $\phi:[a,b]\to \mathbb{R}$ such that
  $|\phi''|\ge1$ on $[a,b]$,
  \begin{equation}\label{eq:van}
    \left|
    \int_{a}^{b}
      e^{i \phi(\lambda)t}\psi(\lambda)d \lambda
    \right|
    \le 10 |t|^{-1/2}
    \left[
      \psi(b)+
      \int_{a}^{b}|\psi'(\lambda)|d \lambda
    \right]
  \end{equation}
  Consider the $+$ case in \eqref{eq:vdc}; we can assume $t>0$.
  After the change of variables
  \begin{equation*}
    \mu=\rho_{+}(\lambda)\implies \lambda=\mmu,\quad
    \lambda\ d \lambda=\mu d\mu
  \end{equation*}
  the integral becomes
  \begin{equation*}
    I_{+}=\int_{0}^{\sqrt{b^{2}-1}}\frac{\mu}{\mu+\mmu}
        e^{i \mu x}
        e^{i\mmu y}
        e^{i(\mu^{2}+1)t}
        d\mu.
  \end{equation*}
  We apply now \eqref{eq:van} with the choice
  \begin{equation*}
    \phi(\mu)=\mu^{2}+1+\mu \frac x t + \mmu \frac y t
    \qquad \implies \qquad
    \phi''=2+\frac{1}{\mmu^{3}} \frac y t\ge 2
  \end{equation*}
  and we obtain
  \begin{equation*}
    |I_{+}|\le
    \frac{10}{|t|^{1/2}}
    \left[
      \frac{\sqrt{b^{2}-1}}{b+\sqrt{b^{2}-1}}+
      \int_{0}^{\sqrt{b^{2}-1}}
      \left|
      \left(
      \frac{\mu}{\mu+\mmu}
      \right)'
      \right|d\mu
      \right].
  \end{equation*}
  Noticing that
  \begin{equation*}
    \left(
    \frac{\mu}{\mu+\mmu}
    \right)'=
    \frac{1}{\mmu(\mu+\mmu)^{2}}
  \end{equation*}
  we arrive at \eqref{eq:vdc}.
  
  Consider now the case of a minus sign in \eqref{eq:vdc}. 
  We take as a phase function
  \begin{equation*}
    \phi=-\lambda^{2}+\frac y t\lambda +\frac x t \sqrt{\lambda^{2}-1}
    \qquad\implies \qquad
    \phi''=-2-\frac{1}{(\lambda^{2}-1)^{3/2}}\frac x t\le-2
  \end{equation*}
  (the singularity of the phase at $\lambda=1$ can
  obviously be overcome by restricting on an interval
  $[1+\epsilon,b]$ and letting $\epsilon \downarrow0$, since the
  final estimate is uniform in $\epsilon$).  As above we obtain
  \begin{equation*}
    |I_{-}|\le
    \frac{10}{|t|^{1/2}}
    \left[
    \frac{b}{b+\sqrt{b^{2}-1}}
    +\int_{1}^{b}
    \left|
    \left(
    \frac{\lambda}{\lambda+\sqrt{\lambda^{2}-1}}
    \right)'\right|d \lambda
    \right].
  \end{equation*}
  Since
  \begin{equation*}
    \left(
    \frac{\lambda}{\lambda+\sqrt{\lambda^{2}-1}}
    \right)'=
    \frac{1}
      {\sqrt{\lambda^{2}-1}(\lambda+\sqrt{\lambda^{2}-1})^{2}}
  \end{equation*}
  we conclude again with \eqref{eq:vdc}.
\end{proof}

Now in order to estimate $I$
(or $III$) it is sufficient to write it as
\begin{equation*}
  I=F *_{x} G
\end{equation*}
with
\begin{equation*}
  F=
  \int_{0}^{8M}\frac{\lambda}{\lambda+\rho_{+}}
  e^{i \rho_{+} x}
  e^{i\lambda y}
  e^{i \lambda^{2}  t}d \lambda,\qquad
  G=\int e^{i \lambda x}\chi(\lambda/M)d \lambda
\end{equation*}
and apply Young's inequality and \eqref{eq:vdc}
\begin{equation*}
  |I|\le \|F\|_{L^{\infty}_{x,y}}\|G\|_{L^{1}_{x}}\le
     15|t|^{-1/2}\|\chi\|_{L^{1}}.
\end{equation*}

For the remaining piece $II$ we need a refinement of the van der Corput
Lemma \eqref{eq:van} which is proved as follows. 
If the function $\psi(\lambda)$ is real valued and
monotonic on $[a,b]$,
we can rewrite \eqref{eq:van} in the equivalent form
\begin{equation*}
  \left|
  \int_{a}^{b}
    e^{i \phi(\lambda)t}\psi(\lambda)d \lambda
  \right|
  \le 10 |t|^{-1/2}
  \left[
    \psi(b)+
    \left|
    \int_{a}^{b}\psi'(\lambda)d \lambda\right|
  \right]\le
  30\|\psi\|_{L^{\infty}}|t|^{-1/2}.
\end{equation*}
As a consequence, if the phase $\phi(\lambda)$
satisfies $|\phi''|\ge1$ on $(a,b)$ and
if both $\Im \psi$ and $\Re \psi$
make at most $N$ oscillations on
the interval $[a,b]$, the following estimate holds
\begin{equation}\label{eq:van2}
  \left|
  \int_{a}^{b}
    e^{i \phi(\lambda)t}\psi(\lambda)d \lambda
  \right|
  \le 60 N \|\psi\|_{L^{\infty}} |t|^{-1/2}
\end{equation}
independently of $a,b$.

Cnsider now the term $II$:
\begin{equation*}
  II=\int_{-1}^{+1} \frac{\lambda}{\rho_{+}+\lambda}
  e^{i \rho_{+}x}e^{i \lambda y}e^{it \lambda^{2}}
  d \lambda,
\end{equation*}
recalling that in the range $-1<\lambda<1$ we have 
$\rho_{+}(\lambda)=i(1-\lambda^{2})^{1/2}$ and
$\chi_{M}=1$.
We shall use estimate \eqref{eq:van2} with the choices
\begin{equation*}
  \phi=\lambda^{2}+\lambda \frac{y}{t}
  \ \implies\ \phi''\ge2
\end{equation*}
and
\begin{equation*}
  \psi=\frac{\lambda}{\rho_{+}+\lambda}e^{i \rho_{+}x}=
     (\lambda^{2}-i \lambda (1-\lambda^{2})^{1/2}))
     e^{-(1-\lambda^{2})^{1/2} x}.
\end{equation*}
It is trivial to check that both the real and the imaginary
part of $\psi$ make at most 2 oscillations on $[-1,1]$,
while $|\psi|\le1$, and in conclusion
\begin{equation*}
  |II|\le 120|t|^{-1/2}.
\end{equation*}


\subsection{Third case: $0<y<x$}\label{sub:third_case}  

Apart from a factor $2i$,
the quantity \eqref{eq:ptw} in this case is the sum of two terms
\begin{equation*}
  w(t,A)=\int \frac{\lambda}{\rho_{+}}
  e^{i \rho_{+}A}e^{it \lambda^{2}}
  \chi_{M}d \lambda,\qquad
  A=x-y>0
\end{equation*}
and
\begin{equation*}
  v(t,A)=\int \frac{\lambda}{\rho_{+}}
  \frac{\rho_{+}-\lambda}{\rho_{+}+\lambda}
  e^{i \rho_{+}A}e^{it \lambda^{2}}
  \chi_{M}d \lambda,\qquad
  A=x+y>0.
\end{equation*}
Now $w(t,A)$ solves the modified Schr\"odinger equation
\begin{equation*}
  iw_{t}-w_{AA}+\delta^{2}w=0 \qquad \implies \qquad
  |w|=|e^{-i\delta^{2}t}w|\le|t|^{-1/2}\|w(0,A)\|_{L^{1}_{A}},
\end{equation*}
and hence we are reduced to give a uniform $L^{1}_{A}$ bound
of the integral
\begin{equation}\label{eq:IIIa0}
  I_{0}=\int
  \frac{\lambda}{\rho_{+}}
  e^{i \rho_{+}A}
  \chi_{M}d \lambda.
\end{equation}
We notice that for $-1<\lambda<1$ the function 
$\rho_{+}(\lambda)=i(1-\lambda^{2})^{1/2}$ is even and so is
$\chi_{M}(\lambda)$, hence the part of the integral on
$[-1,1]$ vanishes. Performing the change of variables
\begin{equation*}
  \mu=\rho_{+}(\lambda)\quad \implies \quad
  \lambda=\sgn\mu\cdot\mmu, \qquad
  \frac{\lambda}{\rho_{+}}d \lambda=d\mu
\end{equation*}
in the remaining parts we get
\begin{equation}\label{eq:II00}
  I_{0}=
  \int_{-\infty}^{-1}+\int_{1}^{+\infty}=
  \int_{-\infty}^{+\infty}
  e^{i \mu A}
  \chi_{M}(\mmu)d \lambda
\end{equation}
since $\chi$ is an even function. Writing 
$\chi(\mu)=\widetilde{\chi }(\mu^{2})$, 
$\widetilde{\chi }\in C^{\infty}_{c}$, we have then
\begin{equation*}
  \|I_{0}\|_{L^{1}_{A}}=
  \left\|
  \int
  e^{i \mu A}
  \widetilde{\chi }(\mu^{2}+M^{-2})d\mu
  \right\|_{L^{1}_{A}}\le
  \pi 
  \int
  \left|
  (1-\partial^{2}_{\mu})
  \widetilde{\chi }(\mu^{2}+M^{-2})
  \right|d\mu
\end{equation*}
which is bounded independently of $M>1$.

The second term $v(t,A)$ can be estimated directly using a
stationary phase argument. For the region $\lambda>1$ we change
variable as above
\begin{equation*}
  \int_{1}^{+\infty} \frac{\lambda}{\rho_{+}}
  \frac{\rho_{+}-\lambda}{\rho_{+}+\lambda}
  e^{i \rho_{+}A}e^{it \lambda^{2}}
  \chi_{M}d \lambda=
  e^{it}
  \int_{0}^{+\infty}
  \frac{\mu-\mmu}{\mu+\mmu}e^{i(\mu^{2}+A/t)t}\chi_{M}(\mmu)d\mu
\end{equation*}
and we can apply the strong form of van der Corput \eqref{eq:van2}
since the phase staisfies
\begin{equation*}
  \phi(\mu)=\mu^{2}+\frac At \quad\implies\quad
  \phi''\ge2
\end{equation*}
while the amplitude
\begin{equation*}
  \frac{\mu-\mmu}{\mu+\mmu}\chi_{M}(\mmu)=
  (2\mu\mmu-1-2\mu^{2})\chi_{M}(\mmu)
\end{equation*}
is bounded by 1 and makes a finite number of oscillations
on $(1,+\infty)$. Thus this part of $v(t,A)$ decays like
$|t|^{-1/2}$, uniformly in $M,A$.
An analogous argument gives the same bound for the
region $\lambda<-1$. For remaining part of the integral
on $-1<\lambda<1$, noticing that the cutoff is equal to 1
there, we split further into the piece
\begin{equation*}
  \int_{0}^{+1} \frac{\lambda}{\rho_{+}}
  \frac{\rho_{+}-\lambda}{\rho_{+}+\lambda}
  e^{i \rho_{+}A}e^{it \lambda^{2}}d\lambda
\end{equation*}
and a symmetric one for $-1<\lambda<0$ which is estimated in
an identical way. Changing variable as $\lambda=\sqrt{1-\mu^{2}}$ 
the integral becomes
\begin{equation*}
  -ie^{it}\int_{0}^{1}
  \frac{\mu-i\sqrt{1-\mu^{2}}}{\mu+i\sqrt{1-\mu^{2}}}
  e^{-A \sqrt{1-\mu^{2}}}e^{-it \mu^{2}}d\mu
\end{equation*}
and we can apply again \eqref{eq:van2} choosing as phase
$\phi=-\mu^{2}$ and as amplitude
\begin{equation*}
  \frac{\mu-i\sqrt{1-\mu^{2}}}{\mu+i\sqrt{1-\mu^{2}}}
  e^{-A \sqrt{1-\mu^{2}}}=
  (2\mu^{2}-1-2i\mu \sqrt{1-\mu^{2}})e^{-A \sqrt{1-\mu^{2}}}.
\end{equation*}
Notice indeed that the amplitude is bounded since $A>0$,
and both its real and imaginary part make a finite number
of oscillations on $(0,1)$ independent of $A$.
Thus also the last piece decays as $|t|^{-1/2}$
and the proof is concluded.



\section{General step potentials}\label{sec:general_step_pot}  

Consider now an $L^{1}$ perturbation of $H_{0}$ of the form
($\one_{+}=\one_{[0,+\infty)}$)
\begin{equation*}
  H=-\frac{d^{2}}{dx^{2}}+\onep(x)+V(x),\qquad
  V\in L^{1}(\mathbb{R})\quad\text{real valued,}
  \qquad
  D(H)=H^{2}(\mathbb{R}).
\end{equation*}
By the standard theory (see e.g. 
\cite{Weidmann80-a},
\cite{Weidmann87-a},
\cite{Teschl09-a})
$H$ is a selfadjoint operator, and its spectrum decomposes as
\begin{equation*}
  \sigma(H)=\sigma_{p}(H)\cup \sigma_{ac}(H),\qquad
  \sigma_{ac}(H)=[0,+\infty),\qquad
  \sigma_{p}=\{\lambda_{j}\}_{j\ge1}\subset(-\infty,0].
\end{equation*}
The sequence of negative eigenvalues, if infinite, can accumulate only
at 0.

\subsection{Low frequencies}\label{sub:low_frequencies}  

From now on we make the stronger assumption on $V$
\begin{equation}\label{eq:VL11}
  (1+x^{2})\cdot V(x)\in L^{1}(\mathbb{R})
\end{equation}
(although for several results  the weaker condition
$(1+|x|)V\in L^{1}$ would be sufficient).
Then most of the standard theory of Jost solutions (see
\cite{DeiftTrubowitz79-a}) carries through to the case of
step potentials, as proved in \cite{CohenKappeler85-a}. We recall
the essential
facts that we shall need in the following. 

Consider the resolvent equation on $\mathbb{R}$
\begin{equation}\label{eq:reseq}
  -\frac{d^{2}}{dx^{2}} f(z,x)+(\onep+V)f(z,x)
  =z^{2}f(z,x),\qquad
  z\in \mathbb{C},\quad \Im z\ge0.
\end{equation}
In order to describe the Jost solutions we extend the definition
of the function $\rho_{+}(z)$ to the upper half plane
$\{\Im z\ge0\}$ as
\begin{equation}\label{eq:rhogen}
  \rho_{+}(z)=\ \text{the branch
  of $(z^{2}-1)^{1/2}$ with nonnegative imaginary part.}
\end{equation}
Notice that for real $z$ this reduces precisely to
\eqref{eq:rho}. The function $\rho_{+}(z)$ is
continuous on $\{\Im z\ge0\}$ and analytic on $\{\Im z>0\}$,
and is a bijection of $\{\Im z>0\}$ onto the upper half plane
with a slit
\begin{equation*}
  \Omega=\{\Im z>0\} \setminus S,\qquad
      S=\{z\in \mathbb{C} \colon
         \Re z=0,\ 0\le \Im z\le1\}
\end{equation*}
with a jump across the slit $S$. Then, under assumption \eqref{eq:VL11},
for each $z$ with $\Im z\ge0$,
equation \eqref{eq:reseq} has two solutions $f_{\pm}(z,x)$
uniquely determined by the properties
\begin{equation}\label{eq:jostprop}
  e^{-i \rho_{+}(z)x}f_{+}(z,x)\to 1 \quad\text{and}\quad
  e^{-i \rho_{+}(z)x}f'_{+}(z,x)\to i \rho_{+}(z)\quad
  \ \text{as $x\to+\infty$},
\end{equation}
\begin{equation}\label{eq:jostprop2}
  e^{i z x}f_{-}(z,x)\to 1 \quad\text{and}\quad
  e^{i z x}f'_{-}(z,x)\to -i z\quad
  \ \text{as $x\to-\infty$}.
\end{equation}
Such solutions are called the \emph{Jost solutions} of the resolvent
equation \eqref{eq:reseq}.
Their Wronskian
\begin{equation}\label{eq:wronskian}
  W(z)=W[f_{+},f_{-}]=
  f_{+}(z,0)\partial_{x}f_{-}(z,0)-
    f_{-}(z,0)\partial_{x}f_{+}(z,0)
\end{equation}
is continuous on $\{\Im z\ge0\}$, analytic 
on $\{\Im z>0\}$,
and satisfies the fundamental property
\begin{equation}\label{eq:Wneq0}
  W(\lambda)\neq0 \quad\text{for $0\neq \lambda\in \mathbb{R}$.}
\end{equation}
According to the standard terminology,
when the Jost solutions are independent at $\lambda=0$ i.e.~when
$W(0)\neq0$, the potential $\one_{+}+V$ 
is said to be of \emph{generic type},
while in the case $W(0)=0$ it is said to be of \emph{exceptional type}.

\begin{theorem}[Lemma 2.4 in \cite{CohenKappeler85-a}]\label{the:Wlambda}
  Assume the potential $V(x)$ satisfies
  \begin{equation}\label{eq:VL12}
    \xx^{2} V(x)\in L^{1}(\mathbb{R}).
  \end{equation}
  Then the Wronskian $W(z)=[f_{+}(z,x),f_{-}(z,x)]$
  is continuous for $\Im z\ge0$, analytic for $\Im z>0$,
  and different from zero for all $z\in \mathbb{R}\setminus0$.
  Moreover, 
  \renewcommand{\labelenumi}{\textit{(\roman{enumi})}}
  \begin{enumerate}
    \item either $W(0)\neq0$,
    \item or $W(0)=0$ and for some real $\gamma\neq0$
    \begin{equation}\label{eq:Wz}
      \dot W(0)=\lim_{z\to0\atop\Im z\ge0}
      \frac{W(z)}{z}=
      i \gamma.
    \end{equation}
  \end{enumerate}
\end{theorem}

Now consider the resolvent equation $(H-z^{2})u=h$ for $h\in L^{2}$
and $z^{2}\not\in \mathbb{R}^{+}$ i.e.~$\Im z>0$. The solution
$u=R(z^{2})h\in L^{2}$ can be expressed by standard ODE theory
using the method of variation of constans, via the kernel
\begin{equation}\label{eq:Kla}
  \mathcal{K}_{z}(x,y)=
  \begin{cases}\displaystyle
    \frac{f_{+}(z,x)f_{-}(z,y)}{W(z)}
     &\text{for $y<x$,}\\
    \mathcal{K}_{z}(y,x)
     &\text{for $y>x$}
  \end{cases}
\end{equation}
as
\begin{equation*}
  R(z^{2})h=(H-z^{2})^{-1}h=\int_{\mathbb{R}}
  \mathcal{K}_{z}(x,y)h(y)dy,
  \qquad
  \Im z>0.
\end{equation*}
Notice that the continuity of $f_{\pm}$ as $z$ approaches the real
axis from positive imaginary values implies that the limit operators
$R(\lambda^{2}+i0)$, $\lambda^{2}\ge0$, with kernel $K_{\lambda}$
given by \eqref{eq:Kla} with $z=\lambda$, are well defined
as operators between suitable weighted $L^{2}$ spaces (i.e.,
the limiting absorption principle holds for $H$). Notice also
that the limits from negative imaginary values are 
given by
\begin{equation*}
  R(\lambda^{2}-i0)=\overline{R(\lambda^{2}+i0)}\quad\text{with
  kernel\ \  $\overline{K_{\lambda}(x,y)}=K_{-\lambda}(x,y)$.}
\end{equation*}
Hence, by the spectral theorem,
fixed any cutoff function 
$\chi(\sqrt{s})\in C^{\infty}_{c}(\mathbb{R}^{+})$
equal to 1 in a neighbourhood of 0,
we can represent the low frequency part of the solution
$e^{itH}f$ as
\begin{equation}\label{eq:lowf}
  P_{ac}e^{itH}\chi(H)g=C
  \int_{-\infty}^{+\infty}
  \int
  e^{it \lambda^{2}}\lambda\chi(\lambda)
    K_{\lambda}(x,y)g(y)dy
      d \lambda.
\end{equation}
Here we have used a change of variables $\lambda\to \lambda^{2}$
in order to express the solution as an integral on the
whole real line. The projection $P_{ac}$
on the absolutely continuous subspace of $H$ is necessary in view
of the possible existence of (negative) eigenvalues.
The precise choice of the cutoff $\chi$ will be made 
later when studying the high frequency
case in Section \ref{sub:high_frequencies}.

In view of \eqref{eq:Kla}, we have the formula
\begin{equation}\label{eq:repreitH}
  P_{ac}e^{itH}\chi(H)g=I+II
\end{equation}
where, apart from inessential constants,
\begin{equation}\label{eq:repreitHI}
  I=
  \int_{y<x}\int_{-\infty}^{+\infty}
  e^{it \lambda^{2}}
  \frac{f_{+}(\lambda,x)f_{-}(\lambda,y)}{W(\lambda)}
  g(y)\lambda \chi(\lambda)d \lambda dy
\end{equation}
and
\begin{equation}\label{eq:repreitHII}
  II=
  \int_{y>x}\int_{-\infty}^{+\infty}
  e^{it \lambda^{2}}
  \frac{f_{+}(\lambda,y)f_{-}(\lambda,x)}{W(\lambda)}
  g(y)\lambda \chi(\lambda)d \lambda dy.
\end{equation}
It is not restrictive to assume that $t>0$ since 
the estimate for $t<0$ can be deduced by conjugation.
Moreover, the two pieces $I$ and $II$ can be handled in a completely
analogous way
so in the following we shall focus on the first term $I$ only,
with $y<x$.

We introduce the standard normalization
\begin{equation}\label{eq:modjost}
  m_{+}(\lambda,x)=e^{-i \rho_{+}(\lambda)x}f_{+}(\lambda,x),\qquad
  m_{-}(\lambda,x)=e^{i \lambda x}f_{-}(\lambda,x)
\end{equation}
so that $m_{\pm}\to 1$ as $\pm x\to+\infty$. 
The functions $m_{\pm}$ are usually called the \emph{Faddeev} solutions.
Notice that the
equations satisfied by $m_{\pm}$ are respectively
($\one_{-}=1-\one_{+}$)
\begin{equation}\label{eq:eqmpm}
  m_{+}''+2i \rho_{+}(\lambda)m_{+}'=(V-\one_{-})m_{+},\qquad
  m_{-}''-2i \lambda m_{-}'=(V+\one_{+})m_{-}.
\end{equation}
Thus our goal is a decay estimate for the integral
\begin{equation}\label{eq:Imod}
  I=
  \int_{y<x}\int_{-\infty}^{+\infty}
  e^{it \lambda^{2}}e^{i \rho_{+}(\lambda)x}e^{-i \lambda y}
  \frac{m_{+}(\lambda,x)m_{-}(\lambda,y)}{W(\lambda)}
  \lambda \chi(\lambda)d \lambda g(y)dy
\end{equation}
and in view of \eqref{eq:dispstepgen}, it is sufficient to prove that
for all $t>0$
\begin{equation}\label{eq:goal}
  \left|
    \int_{-\infty}^{+\infty}
    e^{it \lambda^{2}}e^{i \rho_{+}(\lambda)x}e^{-i \lambda y}
    \frac{m_{+}(\lambda,x)m_{-}(\lambda,y)}{W(\lambda)}
    \lambda \chi(\lambda)d \lambda
  \right|\lesssim t ^{-1/2}
\end{equation}
uniformly in $x,y\in \mathbb{R}$ with $y<x$.

\begin{theorem}\label{the:mpm}
  Assume the potential $V(x)$ satisfies \eqref{eq:VL12}, 
  and let $m_{\pm}$
  the Faddeev solutions defined in \eqref{eq:modjost}.
  Then there exist a constant $C_{V}$ and a
  continuous increasing function 
  $\phi_{V}(x):\mathbb{R}^{+}\to \mathbb{R}^{+}$ such that the following
  holds:
  \renewcommand{\labelenumi}{\textit{(\roman{enumi})}}
  \begin{enumerate}
    \item $m_{-}(\lambda,x)$ is of class $C^{1}$ on 
    $(\lambda,x)\in\mathbb{R}^{2}$
    and satisfies the estimates
    \begin{equation}\label{eq:mmneg}
      |m_{-}(\lambda,x)|+|\partial_{\lambda}m_{-}(\lambda,x)|\le C_{V}
      \quad\text{for $x\le0$,}
    \end{equation}
    \begin{equation}\label{eq:mmpos}
      |m_{-}(\lambda,x)|+|\partial_{\lambda}m_{-}(\lambda,x)|\le 
      \phi_{V}(x)
      \quad\text{for $x\ge0$;}
    \end{equation}
    \item $m_{+}(\lambda,x)$ is continuous om
    $(\lambda,x)\in\mathbb{R}^{2}$, of
    class $C^{1}$ for $\lambda\neq\pm1$, and satisfies the estimates
    \begin{equation}\label{eq:mppos}
      |m_{+}(\lambda,x)|\le C_{V},\qquad
      |\partial_{\lambda}m_{+}(\lambda,x)|\le 
        \frac{|\lambda|C_{V}}{|1-\lambda^{2}|^{1/2}}
      \quad\text{for $x\ge0$,}
    \end{equation}
    \begin{equation}\label{eq:mpneg}
      |m_{+}(\lambda,x)|\le \phi_{V}(x),\qquad
      |\partial_{\lambda}m_{+}(\lambda,x)|\le 
        \frac{|\lambda|\phi_{V}(x)}{|1-\lambda^{2}|^{1/2}}
      \quad\text{for $x\le0$.}
    \end{equation}
    \item 
    More precisely,
    there exists $K\ge0$ such that, for all 
    $\lambda\in \mathbb{R}$
    \begin{equation}\label{eq:asymptm1}
      1-K \sigma_{V}(x)\le 
      m_{+}(\lambda,x)\le 
      1+K\sigma_{V}(x)\qquad
      \ \text{for $x\ge0$}
    \end{equation}
    and
    \begin{equation}\label{eq:asymptm2}
      -K\sigma_{V}(x)\le 
      \partial_{x} m_{+}(\lambda,x)\le 
      K\sigma_{V}(x)\qquad
      \ \text{for all $x\in \mathbb{R}$,}
    \end{equation}
    where
    \begin{equation*}
      \sigma_{V}(x)=\int_{x}^{\infty}(1+|y|)|V(y)|dy.
    \end{equation*}
    The same estimates hold for $m_{-}$, $\partial_{x}m_{-}$
    for all $\lambda\in \mathbb{R}$ and $x\le0$, $x\in \mathbb{R}$
    respectively.
  \end{enumerate}
\end{theorem}

\begin{proof}
  The proof is a reduction to the standard theory for integrable
  potential, as follows.
  For a fixed $M>0$, define the modified potential $V_{M}$ as
  $V_{M}\equiv V+\one_{+}$ for $x\le M$, 
  $V_{M}=0$ for $x>M$; then we have 
  $\xx^{2}V_{M}(x)\in L^{1}(\mathbb{R})$ and the standard theory
  applies. In particular, denoting by $g_{-}(\lambda,x)$ the 
  Jost solution of
  \begin{equation*}
    -g''+V_{M}g=\lambda^{2}g
  \end{equation*}
  with $e^{i \lambda x}g_{-}\to1$ as $x\to-\infty$ and writing
  $n_{-}(\lambda,x)=e^{i \lambda x}g_{-}(\lambda,x)$, we know by Lemma 1 in
  \cite{DeiftTrubowitz79-a} and Lemmas 3.5-3.6 of
  \cite{ArtbazarYajima00} that $n_{-}(\lambda,x)$ is of class
  $C^{1}$ on $\mathbb{R}^{2}$, that $n_{-}$ and $\partial_{\lambda}n_{-}$
  are bounded for $x\le0$, while for $x>0$
  \begin{equation*}
    |n_{-}(\lambda,x)|\le C\xx,\qquad
    |\partial_{\lambda}n_{-}(\lambda,x)|\le C\xx^{2},
  \end{equation*}
  where the constant $C$ depends on the $L^{1}$ norm of $\xx^{2}V_{M}$
  and hence is an increasing function of $M$ only. Since our Jost
  solutions $m_{-}(\lambda,x)$ coincide with $n_{-}(\lambda,x)$
  for $x\le M$, we deduce part (i) immediately. By analysing the 
  proof in \cite{DeiftTrubowitz79-a} one can check that the function
  $\phi_{V}(x)$ grows exponentially, but we shall not need this.
  
  The study of $m_{+}$ is only slightly more difficult. We proceed
  in a similar way: for a fixed $N<0$, we define
  $V_{N}(x)\equiv V(x)-\one_{-}$ for $x\ge N$ (where $\one_{-}$
  is the characteristic function of $\mathbb{R}^{-}$),
  $V_{N}\equiv0$ for $x<N$, and we consider the equation
  \begin{equation*}
    -g''+V_{N}(x)g=(\lambda^{2}-1)g \equiv \rho_{+}(\lambda)^{2}g.
  \end{equation*}
  Notice that this equation coicides with our equation for $f$ in
  the region $x\ge N$, and $\xx^{2}V_{N}\in L^{1}$. As above we
  can apply the standard theory and consider the Jost solution
  $g_{+}(\kappa,x)$, $\kappa=\rho_{+}(\lambda)$,
  uniquely determined by the condition
  $e^{-i \kappa x}g_{+}(\kappa,x)\to1$ as $x\to+\infty$;
  notice however that we must use also complex values of $\kappa$
  since $\rho_{+}(\lambda)$ is pure imaginary for $|\lambda|<1$.
  Thus we define $n_{+}(\kappa,x)=e^{-i \kappa x}g_{+}(\kappa,x)$
  and our Jost solution $m_{+}(\lambda,x)$ satisfies
  \begin{equation*}
    m_{+}(\lambda,x)=n_{+}(\rho_{+}(\lambda),x)
    \quad\text{for $x\ge N$.}
  \end{equation*}
  From the above mentioned Lemmas we 
  easily deduce part (ii) of the Theorem.
  Notice in particular that the boundedness of the first derivative
  in \eqref{eq:mpneg} does not follow directly by the statement
  in Lemma 1 of \cite{DeiftTrubowitz79-a}, which only states
  a quadratic growth, but by an examination of the proof (see
  in particular the last formula on page 135; see also
  \cite{ArtbazarYajima00}).
  
  To prove the final statements \eqref{eq:asymptm1} and
  \eqref{eq:asymptm2}, it is sufficient
  to recall estimates (ii)-(iii) in Lemma 1 of \cite{DeiftTrubowitz79-a}
  concerning the classical Jost function $n_{+}(z,x)$:
  \begin{equation*}
    |n_{+}(z,x)-1|\le K
    \frac{[1+\max(-x,0)]\cdot\int_{x}^{\infty}(1+|y|)|V(y)|dy}{1+|z|}
  \end{equation*}
  and
  \begin{equation*}
    |\partial_{x} n_{+}(z,x)|\le K
    \frac{\int_{x}^{\infty}(1+|y|)|V(y)|dy}{1+|z|}
  \end{equation*}
  which are valid for all $\Im z\ge0$ and $x\in \mathbb{R}$.
  Taking $z=\rho_{+}(\lambda)$ and $x\ge0$ we conclude the proof.
\end{proof}

As above, Van der Corput estimates \eqref{eq:van} will play
an essential role in the following. We collect in a Lemma
some applications that we shall need recurrently:

\begin{lemma}\label{lem:vdcadv}
  Let $a,b,A,B\in \mathbb{R}$ and $h\in C^{1}(a,b)$. Then for
  all $t>0$ the following estimate holds
  \begin{equation}\label{eq:vdcadv}
    \left|
      \int_{a}^{b}
       e^{it \lambda^{2}}e^{i \rho_{+}(\lambda)A}e^{i \lambda B}
       h(\lambda)d \lambda
    \right|\le 30 
    \Bigl[\|h\|_{L^{\infty}(a,b)}+\|h'\|_{L^{1}(a,b)}\Bigr]
    \cdot t^{-1/2}
  \end{equation}
  provided one of the following set of conditions is satisfied:
  \renewcommand{\labelenumi}{\textit{(\roman{enumi})}}
  \begin{enumerate}
    \item $A\le0$, $B\in \mathbb{R}$ and $1\le a<b$; or
    \item $A\in \mathbb{R}$, $B\ge0$ and $1\le a<b$; or
    \item $A\ge0$, $B\in \mathbb{R}$ and $a<b\le1$; or
    \item $A\in \mathbb{R}$, $B\le0$ and $a<b\le-1$.
  \end{enumerate}
\end{lemma}

\begin{proof}
  %
  Case (i): choosing the phase $\phi$ as
  \begin{equation}\label{eq:phasea}
    e^{it \lambda^{2}}e^{i \rho_{+}(\lambda)A}e^{i \lambda B}=
    e^{it \phi(\lambda)},\qquad
    \phi(\lambda)=\lambda^{2}+\rho_{+}(\lambda)\frac At+\lambda
    \frac Bt
  \end{equation}
  we have
  \begin{equation*}
     \phi''=2+\frac At\rho_{+}''\ge2
  \end{equation*}
  since $\rho_{+}''\le0$ on $I=(a,b)\subset(1+\infty)$ and $A\le0$.
  Thus from the standard estimate \eqref{eq:van} 
  we obtain \eqref{eq:vdcadv} (with a numerical constant 10).
  
  Case (ii): under the change of variables $\mu=\rho_{+}(\lambda)$
  i.e.~$\lambda=\mmu$,  the integral becomes
  \begin{equation*}
    e^{it}
    \int_{\rho_{+}(a)}^{\rho_{+}(b)}
    e^{it\mu^{2}}
    e^{i\mu A}
    e^{i\mmu B}
    h(\mmu)\frac{\mu}{\mmu}d\mu
  \end{equation*}
  and we can now choose the phase $\phi$ as
  \begin{equation}\label{eq:phaseb}
    e^{it \mu^{2}}e^{i \mu A}e^{i \mmu B}=
    e^{it \phi(\mu)},\qquad
    \phi(\mu)=\mu^{2}+\mu\frac At+\mmu
    \frac Bt
  \end{equation}
  so that
  \begin{equation*}
    \phi''=2+\frac{1}{\mmu^{3}}\frac Bt\ge0
  \end{equation*}
  since $B\ge0$. This implies \eqref{eq:vdcadv} as before.
  
  Case (iv): we proceed as in (ii), using the change of variables
  $\lambda=-\mmu$ i.e.~$\mu=\rho_{+}(\lambda)$ ($<0$ for $\lambda<-1$),
  and we choose
  \begin{equation}\label{eq:phasec}
    e^{it \mu^{2}}e^{i \mu A}e^{-i \mmu B}=
    e^{it \phi(\mu)},\qquad
    \phi(\mu)=\mu^{2}+\mu\frac At-\mmu
    \frac Bt
  \end{equation}
  again with the property $\phi''\ge2$ since $B\le 0$ now.
  
  Case (iii): we may assume, after possibly splitting the integral
  on two subintervals,
  that we are in one of the two subcases $a<b\le-1$
  or $-1\le a<b\le1$. In the first case we choose the phase
  exactly as in \eqref{eq:phasea}, however $\rho_{+}''\le0$
  for negative $\lambda$ so that now the condition $A\ge0$
  ensures that $\phi''\ge2$ and again we obtain \eqref{eq:vdcadv},
  with a numerical constant 10.
  On the other hand, if $-1\le a<b\le1$, 
  we rewrite the integral in the form
  \begin{equation*}
    \int_{a}^{b}e^{it \phi(\lambda)}g(\lambda)d \lambda,\qquad
    \phi(\lambda)=\lambda^{2}+\lambda \frac Bt,\qquad
    g(\lambda)=h(\lambda)e^{-(1-\lambda^{2})A}.
  \end{equation*}
  By the standard Van der Corput estimate the integral is less than
  \begin{equation*}
    10\left[\|g\|_{L^{\infty}}+\|g'\|_{L^{1}}\right]\cdot t^{-1/2},
  \end{equation*}
  however $|g|\le|h|$ since $A\ge0$, and in addition
  \begin{equation*}
    \|g'\|_{L^{1}}\le\|h'\|_{L^{1}}+\|h\|_{L^{\infty}}
    \int_{a}^{b}\left|\partial_{\lambda}
       e^{-(1-\lambda^{2})A}\right|d \lambda\le 
    \|h'\|_{L^{1}}+2\|h\|_{L^{\infty}}
  \end{equation*}
  since by monotonicity
  \begin{equation*}
    \int_{-1}^{1}\left|\partial_{\lambda}
       e^{-(1-\lambda^{2})A}\right|d \lambda=2-2 e^{-A}\le2.
  \end{equation*}
  Summing up, we obtain \eqref{eq:vdcadv}.
\end{proof}

In order to prove \eqref{eq:goal}, we consider three cases,
according to the relative signs of $y$ and $x$.

\subsubsection{First case: $y<0<x$.}\label{sub:first}  

We split the integral \eqref{eq:goal} in the regions $\lambda>1$
and $\lambda<1$. By the usual change of variables
$\mu=\rho_{+}(\lambda)=(\lambda^{2}-1)^{1/2}$ we can write,
denoting by $\one_{+}$ the characteristic function of 
$\mathbb{R}^{+}$,
\begin{equation*}
\begin{split}
  &\int_{1}^{+\infty}
  e^{it \lambda^{2}}e^{i \rho_{+}(\lambda)x}e^{-i \lambda y}
  \frac{m_{+}(\lambda,x)m_{-}(\lambda,y)}{W(\lambda)}
  \lambda \chi(\lambda)d \lambda=
    \\
  &=e^{it}
  \int_{-\infty}^{\infty}
  \one_{+}(\mu)
  e^{it \mu^{2}}e^{i \mu x}e^{-i \mmu y}
  \frac{m_{+}(\mmu,x)m_{-}(\mmu,y)}{W(\mmu)}
  \mu \chi(\mmu)d \mu
\end{split}
\end{equation*}
which can be interpreted as a Fourier transform
\begin{equation}\label{eq:four}
  e^{it}
  \mathscr{F}_{\mu\to \xi}
  \left.
  \Bigl(
    \one_{+}(\mu)
    e^{it \mu^{2}}e^{-i \mmu y}
    \frac{m_{+}(\mmu,x)m_{-}(\mmu,y)}{W(\mmu)}
    \mu \chi(\mmu)\chi_{1}(\mmu)
  \Bigr)
  \right|_{\xi=x}
\end{equation}
where we inserted an additional even cutoff function $\chi_{1}$
equal to 1 on the support of $\chi$. Writing
\begin{equation*}
  F_{1}(\mu;t,x,y)=
  \one_{+}(\mu)
  e^{it \mu^{2}}e^{-i \mmu y}
  m_{+}(\mmu,x)m_{-}(\mmu,y)\chi_{1}(\mmu)\mu,
\end{equation*}
\begin{equation*}
  F_{2}(\mu)=
  \frac{\chi(\mmu)}{W(\mmu)}
\end{equation*}
the integral \eqref{eq:goal} can be written as a convolution:
\begin{equation*}
  e^{it}\left.
  \widehat{F_{1}}(\xi;t,x,y)*_{\xi}\widehat{F_{2}}(\xi)
  \right|_{\xi=x}.
\end{equation*}
Thus \eqref{eq:goal} will follow from
\begin{equation}\label{eq:goal1}
  \sup_{\xi\in \mathbb{R},\atop y<0<x}
  \left|e^{it}
    \int_{0}^{+\infty}
    e^{it \mu^{2}}e^{-i \mmu y}e^{i\mu\xi}
    m_{+}(\mmu,x)m_{-}(\mmu,y)\chi_{1}(\mmu)\mu d\mu
  \right|\lesssim t^{-1/2}
\end{equation}
and
\begin{equation}\label{eq:goal2}
  \left\|
  \mathscr{F}_{\mu\to\xi}
  \left(
    \frac{\chi(\mmu)}{W(\mmu)}\right)
  \right\|_{L^{1}_{\xi}}<\infty.
\end{equation}
In order to prove \eqref{eq:goal1} we revert to the variable
$\lambda=\mmu$ and obtain the integral
\begin{equation*}
  \int_{1}^{+\infty}
  e^{it \lambda^{2}}e^{-i \lambda y}e^{i\rho_{+}(\lambda)\xi}
  m_{+}(\lambda,x)m_{-}(\lambda,y)\chi_{1}(\lambda)\lambda d \lambda.
\end{equation*}
Using Lemma \ref{lem:vdcadv} (i), this can be estimated by
\begin{equation*}
  30 \sup_{y<0<x}
  \left[
    \|h(\lambda;x,y)\|_{L^{\infty}_{\lambda}}+
    \|\partial_{\lambda} h(\lambda;x,y)\|_{L^{1}_{\lambda}}
  \right]\cdot t^{-1/2}
\end{equation*}
with
\begin{equation*}
  h(\lambda;x,y)=
  m_{+}(\lambda,x)m_{-}(\lambda,y)\chi_{1}(\lambda)\lambda.
\end{equation*}
By Theorem \ref{the:mpm}, recalling that $y<0$ and $x>0$,
we obtain \eqref{eq:goal1}.
On the other hand, \eqref{eq:goal2} follows immediately from the
fact that $W(\lambda)$ is continuous and does not vanish for
real $\lambda$, as stated in Theorem \ref{the:Wlambda}.
This concludes the proof of \eqref{eq:goal} for the region
$\lambda>1$.

As for the $\lambda<1$ piece of \eqref{eq:goal}
\begin{equation*}
  \int_{-\infty}^{1}
  e^{it \lambda^{2}}e^{i \rho_{+}(\lambda)x}e^{-i \lambda y}
  \frac{m_{+}(\lambda,x)m_{-}(\lambda,y)}{W(\lambda)}
  \lambda \chi(\lambda)d \lambda
\end{equation*}
we write it directly as the convolution
\begin{equation*}
  \left.
  \widehat{F_{1}}(\xi;t,x,y)*_{\xi}
  \widehat{F_{2}}(\xi)\right|_{\xi=-y}
\end{equation*}
where the Fourier transform is $\mathscr{F}_{\lambda\to\xi}$,
\begin{equation*}
  F_{1}(\lambda;t,x,y)=\one_{(-\infty,1)}(\lambda)
  e^{it \lambda^{2}}e^{i \rho_{+}(\lambda)x}
  m_{+}(\lambda,x)m_{-}(\lambda,y) \chi_{1}(\lambda)
\end{equation*}
and
\begin{equation*}
  F_{2}(\lambda)=\frac{\lambda\chi(\lambda)}{W(\lambda)}.
\end{equation*}
Now estimate \eqref{eq:goal} follows from an argument 
identical to the previous one, but using
case (iii) of Lemma \ref{lem:vdcadv} instead of case (i),
and the fact that $F_{2}$ is continuous by Theorem \ref{the:Wlambda}.

\subsubsection{Second case: $0<y<x$.}\label{sub:second}  

Consider the region $\lambda>1$ first (the region $\lambda<-1$
is analogous).
The main new difficulty here is that $m_{-}(\lambda,y)$ may
be unbounded as $y\to +\infty$. 
To overcome this problem,
we recall that $f_{-}$ can be expressed as a combination
of $f_{+}(\lambda,x)$ and $\overline{f_{+}(\lambda,x)}$
which for every $\lambda>1$ are two independent solutions
of \eqref{eq:reseq}:
\begin{equation}\label{eq:lincomb}
  f_{-}(\lambda,x)=
  a_{+}(\lambda)
  \overline{f_{+}(\lambda,x)}+b_{+}(\lambda)f_{+}(\lambda,x)
\end{equation}
The quantities $a_{+}$ and $b_{+}$ are computed in
\cite{CohenKappeler85-a} (see formula (1.12) there):
\begin{equation}\label{eq:abm}
  a_{+}(\lambda)=\frac{W(\lambda)}{2i \rho_{+}(\lambda)},\qquad
  b_{+}(\lambda)=\frac{W[\overline{f_{+}(\lambda,x)},f_{-}(\lambda,x)]}
  {2i \rho_{+}(\lambda)}\equiv
  \frac{W_{1}(\lambda)}{2i\rho_{+}(\lambda)}.
\end{equation}
Passing to the functions $m_{+}$ we see that the quantity \eqref{eq:goal}
splits in the sum of two terms:
\begin{equation}\label{eq:goalA}
  \int_{1}^{+\infty}e^{it \lambda^{2}}e^{i\rho_{+}(\lambda)(x-y)}
  m_{+}(\lambda,x)\overline{m_{+}(\lambda,y)}
  \frac{\lambda\chi}{2i \rho_{+}(\lambda)}d \lambda
\end{equation}
and
\begin{equation}\label{eq:goalB}
  \int_{1}^{+\infty}e^{it \lambda^{2}}e^{i\rho_{+}(\lambda)(x+y)}
  m_{+}(\lambda,x){m_{+}(\lambda,y)}
  \frac{\lambda\chi W_{1}(\lambda)}
  {2i \rho_{+}(\lambda)W(\lambda)}d \lambda.
\end{equation}
The first one, after the change of variable $\lambda=\mmu$
and neglecting a factor $e^{it}$, gives
\begin{equation*}
  \left|
  \int_{0}^{\infty}e^{it\mu^{2}}e^{i\mu(x-y)}
  m_{+}(\mmu,x)\overline{m_{+}(\mmu,y)}
  \chi(\mmu)d\mu
  \right|\lesssim t^{-1/2}
\end{equation*}
by Lemma \ref{lem:vdcadv} and the estimates of Theorem \ref{the:mpm};
notice that both $x$ and $y$ are in $\mathbb{R}^{+}$, which is the
good side for $m_{+}$. The second one produces
\begin{equation*}
  \int_{0}^{\infty}e^{it\mu^{2}}e^{i\mu(x+y)}
  m_{+}(\mmu,x){m_{+}(\mmu,y)}
  \frac{\chi(\mmu)W_{1}(\mmu)}{W(\mmu)}d\mu.
\end{equation*}
As we did in the first case (see \eqref{eq:four}), 
we rewrite this integral as the convolution
of two Fourier transforms $\mu\to\xi$
\begin{equation*}
  =
  \left.
  \widehat{F_{1}}(\xi;t,x,y)*_{\xi}\widehat{F_{2}}(\xi)
  \right|_{\xi=x+y}
\end{equation*}
where
\begin{equation*}
  F_{1}(\mu;t,x,y)=\onep(\mu)e^{it \mu^{2}}
  m_{+}(\mmu,x){m_{+}(\mmu,y)}
  \chi_{1}(\mmu),
\end{equation*}
\begin{equation*}
  F_{2}(\mu)=\frac{\mu\chi(\mmu)}{W(\mmu)}
  W_{1}(\mmu)
\end{equation*}
while $\chi_{1}$ is a cutoff equal to 1 on the support of $\chi$.
Now $\widehat{F}_{2}$ is $L^{1}$ since $F_{2}$ is continuous and
compactly supported, and $\widehat{F_{1}}$ is uniformly less than
$Ct^{-1/2}$ by Lemma \ref{lem:vdcadv} and Theorem \ref{the:mpm};

It remains to consider the region $||\lambda|<1$. 
In this case the representation \eqref{eq:lincomb} fails, since
$f_{+}$ is real valued and hence $f_{+}\equiv\overline{f_{+}}$.
Still we can prove a non uniform estimate
as follows. The integral to estimate is now
\begin{equation*}
  \int_{-1}^{1}
  e^{it \lambda^{2}}e^{-(1-\lambda^{2})^{1/2}x}e^{-i \lambda y}
  \frac{m_{+}(\lambda,x)m_{-}(\lambda,y)}{W(\lambda)}
  \lambda \chi(\lambda)d \lambda=
  \left.
    \widehat{F_{1}}(\xi;t,x,y)*_{\xi}\widehat{F_{2}}(\xi)
  \right|_{\xi=y}
\end{equation*}
with
\begin{equation*}
  F_{1}(\lambda;t,x,y)=\one_{[-1,1]}(\lambda)
  e^{it \lambda^{2}}e^{-(1-\lambda^{2})^{1/2}x}e^{-i \lambda y}
  m_{+}(\lambda,x)m_{-}(\lambda,y),
\end{equation*}
\begin{equation*}
  F_{2}(\lambda)=\frac{\lambda\chi}{W(\lambda)}.
\end{equation*}
The Fourier transform $\widehat{F_{2}}$ is $L^1$ by Theorem
\ref{the:Wlambda}, and if we apply Lemma \ref{lem:vdcadv} and
\ref{the:mpm} to $\widehat{F_{1}}$ we obtain
\begin{equation}\label{eq:badest}
  \left|
    \int_{-1}^{+1}
    e^{it \lambda^{2}}e^{-(1-\lambda^{2})^{1/2}x}e^{-i \lambda \xi}
    m_{+}(\lambda,x)m_{-}(\lambda,y)d \lambda
  \right|\lesssim
  t^{-1/2}\cdot
  \phi(y)
\end{equation}
for some continuous function $\phi(y)$. However, in general $\phi$
may grow exponentially and for large values of $y>0$ we need a
different, uniform estimate.

The difficulty here is to get a precise control of
the asymptotic behaviour of the exponentially growing solution 
$f_{-}(\lambda,x)$ for large positive values of $x$.
In the region $x>0$, $-1<\lambda<1$,
the functions $f_{+},f_{-}$ are two independent
solutions of the equation
\begin{equation}\label{eq:auxil}
  f''(\lambda,x)+(\lambda^{2}-1-V(x))f(\lambda,x)=0
\end{equation}
and we know that $f_{+}\sim e^{-(1-\lambda^{2})^{1/2}x}$
is exponentially decreasing. Notice that by \eqref{eq:asymptm1}
there exists $a>0$ such that
\begin{equation}\label{eq:preciseas}
  \frac12\le m_{+}(\lambda,x)\le \frac32
  \quad\text{for\ \  $x>a$,\ \ $\lambda\in \mathbb{R}$.}
\end{equation}
and hence, for $|\lambda|\le1$ and $x\ge a$,
\begin{equation}\label{eq:precisef}
  \frac12 e^{-(1-\lambda^{2})^{1/2}x}
  \le f_{+}(\lambda,x)\le
  \frac32 e^{-(1-\lambda^{2})^{1/2}x}
\end{equation}
Then the function $g_{+}(\lambda,x)$ given by
\begin{equation*}
  g_{+}(\lambda,x)=2(1-\lambda^{2})^{1/2} \cdot f_{+}(\lambda,x)
  \int_{a}^{x}\frac{dy}{f_{+}(\lambda,y)^{2}}
\end{equation*}
is well defined on $x\ge a$, $|\lambda|\le1$ and
is a second solution of the equation \eqref{eq:auxil} there,
a well known fact from the general ODE theory
which can be easily checked directly.

In the following we shall need both the precise asymptotic behaviour
of $g_{+}$ and $\partial_{x}g_{+}$ as $x\to+\infty$, and uniform
estimates. Recall that
(all the asymptotics are for $x\to+\infty$, and we restrict
$\lambda$ to $|\lambda|<1$)
\begin{equation*}
  f_{+}\sim e^{-(1-\lambda^{2})^{1/2}x},\qquad
  \partial_{x}f_{+}\sim-(1-\lambda^{2})^{1/2}
   e^{-(1-\lambda^{2})^{1/2}x}
\end{equation*}
by \eqref{eq:asymptm1}, \eqref{eq:asymptm2}. Then we can write,
using de l'H\^opital's theorem,
\begin{equation*}
  f_{+}\cdot g_{+}=2(1-\lambda^{2})^{1/2}
  \frac{\int_{a}^{x}f_{+}^{-2}}{f_{+}^{-2}}\sim 
  2(1-\lambda^{2})^{1/2}
  \frac{f_{+}}{-2\partial_{x}f_{+}}\sim 
  1
\end{equation*}
by the previous asymptotics, and hence
\begin{equation}\label{eq:asymptg}
  g_{+}(\lambda,x)\sim e^{(1-\lambda^{2})^{1/2}x}.
\end{equation}
On the other hand we have
\begin{equation}\label{eq:asymptg2}
  \partial_{x}g_{+}=
  \frac{\partial_{x}f_{+}}{f_{+}}g_{+}
  +\frac{2(1-\lambda^{2})^{1/2}}{f_{+}}\sim 
  (1-\lambda^{2})^{1/2}
  e^{(1-\lambda^{2})^{1/2}x}.
\end{equation}
Thus we can compute the Wronskian
\begin{equation}\label{eq:wro}
  W[g_{+},f_{+}]=2(1-\lambda^{2})^{1/2}
\end{equation}
and in particular we obtain that $g_{+}, f_{+}$ are linearly
independent.

In order to get uniform estimates, we notice that
\eqref{eq:mppos}, \eqref{eq:mmpos} imply,
for $x\ge0$, $|\lambda|<1$
\begin{equation*}
  |\partial_{\lambda}f_{+}(\lambda,x)|\le
  (C_{V}+|x|)e^{-(1-\lambda^{2})^{1/2}x}
  \frac{|\lambda|}{(1-\lambda^{2})^{1/2}}
\end{equation*}
and
\begin{equation*}
  |\partial_{\lambda}f_{-}(\lambda,x)|\le
  (1+|x|)\phi_{V}(x).
\end{equation*}
We have also, for $x\ge a$,
\begin{equation*}
  \frac2{9(1-\lambda^{2})^{1/2}}
  e^{2(1-\lambda^{2})^{1/2}(x-a)}\le
  \int_{a}^{x}\frac{dy}{f_{+}(\lambda,y)^{2}}\le 
  \frac{2}{(1-\lambda^{2})^{1/2}}
  e^{2(1-\lambda^{2})^{1/2}x}
\end{equation*}
by \eqref{eq:precisef}, and by the definition of $g_{+}$
\begin{equation*}
  \frac19 e^{(1-\lambda^{2})^{1/2}(x-2a)}\le
  g_{+}(\lambda,x)\le 3 e^{(1-\lambda^{2})^{1/2}x}.
\end{equation*}
Also from the definition of $g_{+}$ and the above estimates
it follows easily that
\begin{equation*}
  |\partial_{\lambda}g_{+}(\lambda,x)|\le 
  C(C_{V}+|x|)
  \frac{|\lambda|}{(1-\lambda^{2})}
  e^{(1-\lambda^{2})^{1/2}x}.
\end{equation*}

Now we can express $f_{-}$ as a linear combination
\begin{equation}\label{eq:lincombg}
  f_{-}(\lambda,x)=A(\lambda)g_{+}(\lambda,x)+
     B(\lambda)f_{+}(\lambda,x),\qquad|\lambda|<1,\qquad
     x\ge a.
\end{equation}
Taking the Wronskian with $f_{+}$ and recalling \eqref{eq:wro} 
we obtain
\begin{equation}\label{eq:idwro}
  W(\lambda)\equiv A(\lambda)W[g_{+},f_{+}]
  =2(1-\lambda^{2})^{1/2}A(\lambda).
\end{equation}
We know that $W(\lambda)$ is continuous and does not vanish for
real $\lambda$ (Theorem \ref{the:Wlambda}), so that $A(\lambda)$
can not vanish, is continuous for $|\lambda|<1$, and
must diverge as $\lambda\to\pm1$,
more precisely, for some $C,C'>0$,
\begin{equation}\label{eq:estAl}
  \frac{C}{(1-\lambda^{2})^{1/2}}
  \le
  A(\lambda)\le \frac{C'}{(1-\lambda^{2})^{1/2}}
  \quad\text{on $(-1,1)$.}
\end{equation}
From the definition of $g_{+}$ we see that $g_{+}(\lambda,a)=0$,
hence \eqref{eq:lincombg} implies
\begin{equation*}
  B(\lambda)=\frac{f_{-}(\lambda,a)}{f_{+}(\lambda,a)}=
  \frac{m_{-}(\lambda,a)}{m_{+}(\lambda,a)}
    e^{-i \lambda x}e^{-(1-\lambda^{2})^{1/2}x}.
\end{equation*}
Using \eqref{eq:preciseas},
\eqref{eq:mmpos} and \eqref{eq:mppos} we thus obtain
\begin{equation}\label{eq:Bdot}
  |B(\lambda)|\le C, \qquad
  |\partial_{\lambda}B(\lambda)|\le \frac{C}
        {(1-\lambda^{2})^{1/2}}
  \qquad \ \text{for $|\lambda|<1$.}
\end{equation}
By \eqref{eq:estAl} and \eqref{eq:Bdot} we have
\begin{equation}\label{eq:estBonA}
  \left|
  \frac{B(\lambda)}{A(\lambda)}
  \right|\le 
  C(1-\lambda^{2})^{1/2}\le C
  \qquad \ \text{for $|\lambda|<1$.}
\end{equation}
Moreover we can represent $A(\lambda)$ as
\begin{equation*}
  A(\lambda)=\frac{f_{-}}{g_{+}}-B \frac{f_{+}}{g_{+}}
\end{equation*}
and by differentiating with respect to $\lambda$, using the
previous estimates we obtain easily
\begin{equation}\label{eq:Adot}
  |\partial_{\lambda}A(\lambda)|\le
  \frac{C}{1-\lambda^{2}}
  \qquad \ \text{for $|\lambda|<1$.}
\end{equation}
Finally, using \eqref{eq:estAl}, \eqref{eq:Adot} and \eqref{eq:Bdot}
we see that
\begin{equation}\label{eq:estderBA}
  \left|
    \partial_{\lambda}
    \left(\frac{B(\lambda)}{A(\lambda)}\right)
  \right|\le C
  \qquad \ \text{for $|\lambda|<1$.}
\end{equation}

We come back to the integral we are set to estimate:
\begin{equation*}
  \int_{-1}^{+1}
  e^{it \lambda^{2}}
  f_{+}(\lambda,x)f_{-}(\lambda,y)
  \frac{\lambda d \lambda}{W(\lambda)},\qquad
  0<y<x.
\end{equation*}
When $y\le a$
we can use estimate \eqref{eq:badest} already proved, and
it remains to consider the case $x>y>a$. In this region
the representation \eqref{eq:lincombg} applies and
the integral can be written
\begin{equation*}
  I=\int_{-1}^{+1}
  e^{it \lambda^{2}}
  f_{+}(\lambda,x)
  \left(
    g_{+}(\lambda,y)+\frac{B(\lambda)}{A(\lambda)}
    f_{+}(\lambda,y)
  \right)
    \frac{\lambda}{2(1-\lambda^{2})^{1/2}}d \lambda
\end{equation*}
where we have used identity \eqref{eq:idwro}. Recalling the definition
of $g_{+}$, we see that it is enough to estimate the two
integrals
\begin{equation*}
  I_{1}=\int_{0}^{+1}
  e^{it \lambda^{2}}
  f_{+}(\lambda,x)f_{+}(\lambda,y)
    \frac{B(\lambda)}{A(\lambda)}
    \frac{\lambda d \lambda}{2(1-\lambda^{2})^{1/2}}
\end{equation*}
and
\begin{equation*}
  I_{2}=\int_{0}^{+1}
  e^{it \lambda^{2}}
  f_{+}(\lambda,x)f_{+}(\lambda,y)
    \int_{a}^{y}\frac{ds}{f_{+}(\lambda,s)^{2}},
    \lambda d \lambda
\end{equation*}
since the corresponding integrals on $-1<\lambda<0$ can be handled 
exactly in the same way. 
We rewrite $I_{1}$ in terms of $m_{+}$ and perform the change
of variables $\lambda=(1-\mu^{2})^{1/2}$ to obtain
\begin{equation*}
  I_{1}=e^{it}\int_{0}^{1}e^{-i\mu^{2}t}h_{1}(\mu)d\mu
\end{equation*}
where
\begin{equation*}
  h(\mu)=e^{-\mu(x+y)}
  m_{+}((1-\mu^{2})^{1/2},x)m_{+}((1-\mu^{2})^{1/2},y)
      \frac{B(1-\mu^{2})^{1/2})}{A(1-\mu^{2})^{1/2})}
\end{equation*}
By \eqref{eq:mppos} we have
\begin{equation}\label{eq:estimu}
  \Bigl|m_{+}((1-\mu^{2})^{1/2},x)\Bigr|\le C
\end{equation}
and
\begin{equation}\label{eq:estinmuder}
  \Bigl|\partial_{\mu}m_{+}((1-\mu^{2})^{1/2},x)\Bigr|=
  \Bigl|\left.\partial_{\lambda}m_{+}(\lambda,x)\right|
      _{\lambda=(1-\mu^{2})^{1/2}}\Bigr|\cdot
      \frac{|\mu|}{(1-\mu^{2})^{1/2}}\le C
\end{equation}
for some $C$ independent of $x\ge0$. Moreover by \eqref{eq:estBonA},
\eqref{eq:estderBA} we have
\begin{equation}\label{eq:BonAgood}
  \left|
  \frac{B(1-\mu^{2})^{1/2})}{A(1-\mu^{2})^{1/2})}
  \right|\le C|\mu|
\end{equation}
and
\begin{equation*}
  \left|
  \partial_{\mu}
  \frac{B(1-\mu^{2})^{1/2})}{A(1-\mu^{2})^{1/2})}
  \right|\le C\cdot
  \frac{|\mu|}{(1-\mu^{2})^{1/2}}\in L^{1}(0,1).
\end{equation*}
Finally, for $x,y>0$, we have
\begin{equation*}
  e^{-\mu(x+y)}\le 1,\qquad
  |\partial_{\mu}e^{-\mu(x+y)}|\le(x+y)e^{-\mu(x+y)}\le \frac{1}{|\mu|}
\end{equation*}
and we notice that the $|\mu|^{-1}$ singularity is canceled by
the $|\mu|$ factor from estimate \eqref{eq:BonAgood}.
In conclusion, we see that the amplitude $h_{1}(\mu)$ in $I_{1}$
satisfies
\begin{equation*}
  \|h_{1}\|_{L^{\infty}(0,1)}+\|\partial_{\mu}h_{1}\|_{L^{1}(0,1)}\le C
\end{equation*}
for a $C$ independent of $x,y\ge a$. A standard application of van der
Corput Lemma \eqref{eq:van} gives then 
\begin{equation}\label{eq:I1est}
  |I_{1}|\le C|t|^{-1/2}.
\end{equation}

In order to estimate the second integral $I_{2}$, 
after the same change of variables
$\lambda=(1-\mu^{2})^{1/2}$, we rewrite it in the form
\begin{equation*}
  I_{2}=e^{it}\int_{0}^{1}e^{-it\mu^{2}}h_{2}(\mu)d\mu
\end{equation*}
with
\begin{equation*}
  h_{2}(\mu)=
  m_{+}((1-\mu^{2})^{1/2},x)m_{+}((1-\mu^{2})^{1/2},y)
  \int_{a}^{y}\frac{\mu e^{\mu(2s-x-y)}ds}
  {m_{+}((1-\mu^{2})^{1/2},s)^{2}}.
\end{equation*}
We further split
\begin{equation*}
\begin{split}
  \int_{a}^{y} &
  \frac{\mu e^{\mu(2s-x-y)}ds}
  {m_{+}((1-\mu^{2})^{1/2},s)^{2}}=
    \\
  &=
  \frac12
  e^{\mu(y-x)}
  -\frac12 
  e^{\mu(2a-x-y)}+
  \int_{a}^{y}
   \left[
   \frac{1}
   {m_{+}((1-\mu^{2})^{1/2},s)^{2}}
   -1
   \right]
   \mu e^{\mu(2s-x-y)}ds.
\end{split}
\end{equation*}
(where we have added and subtracted 1 inside the integral).
This gives
\begin{equation*}
  I_{2}=I_{3}+I_{4}+I_{5},\qquad
  I_{j}=e^{it}\int_{0}^{1}e^{-it\mu^{2}}h_{j}(\mu)d\mu
\end{equation*}
where
\begin{equation*}
  h_{3}=\frac12e^{\mu(y-x)}
  m_{+}((1-\mu^{2})^{1/2},x)\cdot m_{+}((1-\mu^{2})^{1/2},y),
\end{equation*}
\begin{equation*}
  h_{4}=
  -\frac12 
  e^{\mu(2a-x-y)}
  m_{+}((1-\mu^{2})^{1/2},x)\cdot m_{+}((1-\mu^{2})^{1/2},y),
\end{equation*}
\begin{equation*}
\begin{split}
  h_{5}=m_{+}((1-\mu^{2})^{1/2},x)\cdot &
       m_{+}((1-\mu^{2})^{1/2},y)\times
    \\
  &\times
  \int_{a}^{y}
   \left[
   \frac{1}
   {m_{+}((1-\mu^{2})^{1/2},s)^{2}}
   -1
   \right]
   \mu e^{\mu(2s-x-y)}ds.
\end{split}
\end{equation*}
The function $h_{3}$ satisfies
\begin{equation*}
  \|h_{3}\|_{L^{\infty}(0,1)}\le C,\qquad
  \|\partial_{\mu} h_{3}\|_{L^{1}(0,1)}\le C
\end{equation*}
with $C$ independent of $x,y>0$; this follows from
\eqref{eq:estimu}, \eqref{eq:estinmuder} and the fact that
\begin{equation*}
  \int_{0}^{1}|\partial_{\mu}e^{\mu(y-x)}|d\mu=
  \left|
    \int_{0}^{1}\partial_{\mu}e^{\mu(y-x)}d\mu
  \right|\le1
\end{equation*}
by monotonicity, since $x>y$. Thus a direct application of
van der Corput Lemma \eqref{eq:van} gives
\begin{equation*}
  |I_{3}|\le C|t|^{-1/2};
\end{equation*}
an identical argument gives 
\begin{equation*}
  |I_{4}|\le C|t|^{-1/2}.
\end{equation*}

Finally we focus on the more difficult term $I_{5}$.
It is easy to check that the function $h_{5}$ is uniformly
bounded, using \eqref{eq:estimu} and the inequality
\begin{equation*}
  \left|
    \int_{a}^{y}\frac{\mu e^{\mu(2s-x-y)}ds}
    {m_{+}((1-\mu^{2})^{1/2},s)^{2}}
  \right|\le 4
  \int_{a}^{y}\mu e^{\mu(2s-x-y)}ds\le2\qquad
  \ \text{for\ \ }
  x>y>a.
\end{equation*}
which follows from \eqref{eq:preciseas}. Next, we need to prove
a uniform bound for $\partial_\mu h_{5}$
in $L^{1}(0,1)$.
We have already seen that all three factors in $h_{5}$
are bounded, and the first two have a uniformly bounded
derivative by \eqref{eq:estimu}, thus it remains
to check that
\begin{equation*}
  \int_{0}^{1}
  \left|
    \partial_{\mu}
    \int_{a}^{y}
     \left[
     \frac{1}
     {m_{+}((1-\mu^{2})^{1/2},s)^{2}}
     -1
     \right]
     \mu e^{\mu(2s-x-y)}ds
  \right|d\mu\le C
\end{equation*}
Expanding the derivative gives the two terms
\begin{equation*}
  P=2\int_{0}^{1}
  \left|
  \int_{a}^{y}\frac{\partial_{\mu}[m_{+}((1-\mu^{2})^{1/2},s)]}
  {m_{+}((1-\mu^{2})^{1/2},s)^{3}}\mu e^{\mu(2s-x-y)}ds
  \right|d\mu
\end{equation*}
and
\begin{equation*}
  Q=\int_{0}^{1}
  \left|\int_{a}^{y}
  \left[
  \frac{1}
  {m_{+}((1-\mu^{2})^{1/2},s)^{2}}
  -1
  \right]
  \partial_{\mu}(\mu e^{\mu(2s-x-y)})ds
  \right|d\mu.
\end{equation*}
Since $m_{+}\ge1/2$ and 
$|\partial_{\mu}[m_{+}((1-\mu^{2})^{1/2},s)]|\le C$, the quantity
$P$ is bounded by
\begin{equation*}
  P\le C\int_{0}^{1}\int_{a}^{y}\mu e^{\mu(2s-x-y)}ds\,d\mu
  \le C'
  \quad\text{for $x,y>a$.}
\end{equation*}
In order to bound $Q$, we first notice that
\begin{equation*}
  \left|
  \frac{1}
  {m_{+}((1-\mu^{2})^{1/2},s)^{2}}
  -1
  \right|\le 
  C\left|1-m_{+}((1-\mu^{2})^{1/2},s)^{2}\right|\le 
  C'\sigma_{V}(s)
\end{equation*}
by \eqref{eq:asymptm1}, where
\begin{equation*}
  \sigma_{V}(s)=\int_{s}^{\infty}(1+|\xi|)|V(\xi)|d\xi.
\end{equation*}
This implies
\begin{equation*}
  Q\lesssim \int_{0}^{1}\int_{a}^{y}
    \sigma_{V}(s)\cdot|1+\mu(2s-x-y)|\cdot e^{\mu(2s-x-y)}ds\,d\mu.
\end{equation*}
Exhanging the order of integration and changing variables with
$s=(x+y-r)/2$, $\mu=\nu/r$, we obtain
\begin{equation*}
  =\frac12\int_{x-y}^{x+y-2a}
  \sigma_{V}\left(\frac{x+y-r}{2}\right)
  \int_{0}^{r}|1-\nu|e^{-\nu}d\nu \frac{dr}{r}.
\end{equation*}
Since
\begin{equation*}
  \frac1r\int_{0}^{r}|1-\nu|e^{-\nu}d\nu\le \frac{C}{1+r}\le C
  \quad\text{for $r>0$,}
\end{equation*}
we have the estimate
\begin{equation*}
\begin{split}
  Q \lesssim &
  \int_{x-y}^{x+y}\sigma_{V}
  \left(\frac{x+y-r}{2}\right)dr=
  \int_{0}^{y}\sigma_{V}(s)ds\le 
  \int_{0}^{+\infty}\sigma_{V}(s)ds
    \\
  =&\int_{0}^{+\infty}\int_{s}^{+\infty}(1+|\xi|)|V(\xi)|d\xi\,ds
  \le \int_{0}^{+\infty}(1+|\xi|)^{2}|V(\xi)|d\xi <\infty
\end{split}
\end{equation*}
by the assumption on $V$. In conclusion, 
$\|\partial_{\mu}h_{5}\|_{L^{1}}$ is uniformly bounded and we obtain
\begin{equation*}
  |I_{5}|\le C|t|^{-1/2}
\end{equation*}
and the proof of this case is concluded.

\subsubsection{Third case: $y<x<0$.}\label{sub:third}  

The proof is similar to the second case but easier.
In the integral \eqref{eq:goal}, the troublesome
factor is now
$f_{+}(\lambda,x)$ which can be expressed for all $\lambda$
using formulas (1.10)-(1.13) in \cite{CohenKappeler85-a} as
\begin{equation}\label{eq:lincomb2}
  f_{+}(\lambda,x)=
  a_{-}(\lambda)
  \overline{f_{-}(\lambda,x)}+b_{-}(\lambda)f_{-}(\lambda,x)
\end{equation}
where
\begin{equation}\label{eq:abp}
  a_{-}(\lambda)=\frac{W(\lambda)}{2i \lambda},\qquad
  b_{-}(\lambda)=\frac{W[f_{+}(\lambda,x),\overline{f_{-}(\lambda,x)}]}
  {2i \lambda}\equiv
  \frac{W_{2}(\lambda)}{2i\lambda}.
\end{equation}
Then \eqref{eq:goal} splits in the sum of the two terms
\begin{equation}\label{eq:goalA1}
  \int_{-\infty}^{+\infty}
  e^{it \lambda^{2}}e^{i \lambda(x-y)}
  \overline{m_{-}(\lambda,x)}m_{-}(\lambda,y)
  \frac{\chi(\lambda)}{2i}d \lambda
\end{equation}
and
\begin{equation}\label{eq:goalB1}
  \int_{-\infty}^{+\infty}
  e^{it \lambda^{2}}e^{-i \lambda(x+y)}
  m_{-}(\lambda,x)m_{-}(\lambda,y)
  \frac{\chi(\lambda)W_{2}(\lambda)}{2i}d \lambda
\end{equation}
which can be estimated exactly as \eqref{eq:goalA} and 
\eqref{eq:goalB} above; it is not necessary to
handle the region $|\lambda|<1$ any differently since
\eqref{eq:lincomb2} is available for all $\lambda$.


\subsection{High frequencies}\label{sub:high_frequencies}  


In this section we study the part of the solution corresponding
to high frequences
\begin{equation}\label{eq:high}
  P_{ac}(1-\chi(H))e^{itH}f=
  \int e^{it \lambda^{2}} R(\lambda^{2}+i0)f \lambda \psi(\lambda)
  d\lambda,\qquad
  \psi=1-\chi
\end{equation}
where $\psi(\lambda)=1-\chi(\lambda)$ vanishes for 
$|\lambda|\le \lambda_{0}$, $\lambda_{0}$ to be chosen.
Notice that the following argument requires only 
$V\in L^{1}(\mathbb{R})$.

The resolvent $R$ for the operator 
$H=-d^{2}_{x}+V(x)+\one_{+}$ and the resolvent $R_{0}$ of
the operator $-d^{2}_{x}+\one_{+}$
are related by the standard identity
\begin{equation*}
  R_{0}(z)=(I+R_{0}V)R,
\end{equation*}
which can be formally expanded to
\begin{equation*}
  R(\lambda^{2}+i0)=
  \sum_{k\ge0}(-1)^{k}(R_{0}(\lambda^{2}+i0)V)^{k}R_{0}(\lambda^{2}+i0).
\end{equation*}
We represent $k$-th term of the series
using the explicit espression \eqref{eq:Kl}
of the free kernel $K_{\lambda}$ as
\begin{equation*}
  (R_{0}V)^{k}R_{0}f=\int 
    K_{\lambda}(x,y_{0})V(y_{0})
    K_{\lambda}(y_{0},y_{1})V(y_{1})
    \dots
    K_{\lambda}(y_{k-1},y_{k})f(y_{k})
    dy_{0}\dots dy_{k}
\end{equation*}
and this leads to the representation
\begin{equation}\label{eq:expan}
  P_{ac}(1-\chi(H))e^{itH}f=\sum_{k\ge0}(-1)^{k}A_{k}f
\end{equation}
where
\begin{equation}\label{eq:Ak}
  A_{k}f=
  \int dy_{0}\cdot\int dy_{k}
  V(y_{0})\dots V(y_{k})
  \gamma_{k}(\lambda;t,x,y_{0},\dots,y_{k})f(y_{k})
\end{equation}
and
\begin{equation}\label{eq:gak}
  \gamma_{k}(\lambda;t,x,y_{0},\dots,y_{k})
  =\int e^{it \lambda^{2}}
  K_{\lambda}(x,y_{0})\dots K_{\lambda}(y_{k-1},y_{k})
  \lambda \chi(\lambda)
  d \lambda.
\end{equation}
We shall prove that, if $\lambda_{0}$ is large enough that
$\rho_{+}(\lambda_{0})>1$, we have
\begin{equation}\label{eq:estgak}
  |\gamma_{k}|\le C \cdot t^{-1/2}\rho_{+}(\lambda_{0})^{-k}
\end{equation}
with a constant independent of $k,y_{0},\dots,y_{k}$. 
By \eqref{eq:Ak} this implies
\begin{equation*}
  |A_{k}f|\le C \cdot t^{-1/2}\|V\|_{L^{1}}\rho_{+}(\lambda_{0})^{-k}
\end{equation*}
and hence, choosing
\begin{equation}\label{eq:choicel}
  \lambda_{0}=
  \rho_{+}^{-1}(2\|V\|_{L^{1}}+2)
\end{equation}
we obtain at the same time the convergence of the expansion
\eqref{eq:expan} for $t>0$ and the claimed decay estimate for the 
solution.

Thus let us focus on proving \eqref{eq:estgak}. 
The term $A_{0}$ coincides with the expression of the
solution when $V\equiv 0$, so the estimate we need was already
proved in Section \ref{sec:proof1}. Thus let us consider the
terms $A_{k}$ with $k\ge1$ (with some additional care necessary
when $k=1$, see the end of the proof). By 
examining the explicit expression \eqref{eq:Kl} we see
that the product of kernels $K_{\lambda}$ has the form
\begin{equation}\label{eq:prodKl}
  K_{\lambda}(x,y_{0})\dots K_{\lambda}(y_{k-1},y_{k})=
  \sum \frac{1}{(2i)^{k+1}}
  e^{i \lambda A}e^{i \rho_{+}B}\sigma(\lambda)
\end{equation}
where:
\begin{enumerate}
  \item the number of terms in the sum does not exceed $2^{k+1}$;
  \item the quantities $A,B$ are 
  linear combinations of $x,y_{0},\dots,y_{k}$
  with the property that
  \begin{equation*}
    A\ge0,\qquad B\ge0;
  \end{equation*}
  \item the functions $\sigma(\lambda)$ are products of the form
  \begin{equation*}
    \sigma(\lambda)=
    \frac1 {\lambda^{\ell}}
    \left(\frac{1}{\lambda}
         \frac{\lambda-\rho_{+}}{\lambda+\rho_{+}}\right)^{m}
    \frac1 {(\lambda+\rho_{+})^{n}}
    \frac1 {\rho_{+}^{p}}
    \left(\frac{1}{\rho_{+}}
         \frac{\lambda-\rho_{+}}{\lambda+\rho_{+}}\right)^{q}
  \end{equation*}
  where the non negative integers $\ell,m,n,p,q$ satisfy
  \begin{equation*}
    \ell+m+n+p+q=k+1.
  \end{equation*}
\end{enumerate}
We elaborate a little on the properties of the functions $\sigma$.
Notice that we are in the region $|\lambda|\ge \lambda_{0}>1$. Using
the identity $\lambda-\rho_{+}=(\lambda+\rho_{+})^{-1}$
we rewrite the expression of $\sigma$ in the simpler form
\begin{equation*}
  \sigma(\lambda)=
  \frac1 {\lambda^{\ell+m}}
  \frac1 {\rho_{+}^{p+q}}
  \frac1 {(\lambda+\rho_{+})^{n+2m+2q}}.
\end{equation*}
In particular, we notice that $\sigma(\lambda)$ is monotone decreasing
on $\lambda>1$ and monotone increasing on $\lambda<-1$, and for
$|\lambda|\ge \lambda_{0}$ we have (since $|\lambda|\ge |\rho_{+}|$
and $\rho_{+}(\lambda_{0})>1$)
\begin{equation}\label{eq:boundsig}
  |\sigma(\lambda)|\le \rho_{+}(\lambda_{0})^{-(\ell+p+n+3m+3q)}
  \le \rho_{+}(\lambda_{0})^{-k-1}.
\end{equation}
Thus $\sigma(\lambda)$ is monotone and satisfies the bound 
\eqref{eq:boundsig}; when $k\ge2$,
a direct application of Lemma \ref{lem:vdcadv},
keeping into account that $A,B\ge0$ and also
the additional factor $\lambda$ (which is bounded by
$2 \rho_{+}(\lambda)$ with our choice of $\lambda_{0}$)
gives
\begin{equation*}
  \left|
    \int e^{it \lambda^{2}e^{i \lambda A}e^{i \rho_{+}(\lambda)B}}
    \sigma(\lambda)\lambda\chi(\lambda)d \lambda
  \right|\le 
  240 \rho_{+}(\lambda_{0})^{-k}.
\end{equation*}
Summing the estimates over all the terms in \eqref{eq:prodKl}
and noticing the power of 2 at the denominator, we conclude
the proof of \eqref{eq:estgak} with a constant $C=240$.

In the case $k=1$ there is an additional 
technical difficulty due to the
fact that the convergence of the integral in $\lambda$
must be justified since the integrand decays like $\lambda^{-1}$
only. To this end
it is sufficient to approximate $A_{1}f$ by introducing
an additional cutoff of the form $\chi(\lambda/M)$ and noticing
that the estimate is uniform as $M\to \infty$.

The proof of the high energy case is concluded.




\appendix  
\section{Two lemmas}\label{sec:appendix}

\begin{lemma}\label{lem:phase1}
  For all $\chi\in C^{1}_{c}(\mathbb{R})$ and $A\neq0$,
  \begin{equation}\label{eq:phase1}
    \left|
    \int_{0}^{\infty}e^{i \lambda A}\chi(\lambda)\lambda^{-1/2}d \lambda
    \right|\le
    5\|\chi'\|_{L^{1}}\cdot |A|^{-1/2}.
  \end{equation}
\end{lemma}

\begin{proof}
  Notice the inequality
  \begin{equation*}
    \|\chi\|_{L^{\infty}}\le\|\chi'\|_{L^{1}}
  \end{equation*}
  since $\chi$ is compactly supported. We split the integral as
  $\int_{0}^{\epsilon}+\int_{\epsilon}^{+\infty}$ and we calculate
  \begin{equation*}
    \left|\int_{0}^{\epsilon}\right|\le
    \|\chi\|_{L^{\infty}}
    \int_{0}^{\epsilon}\lambda^{-1/2}d \lambda\le
    2 \epsilon^{1/2}\|\chi'\|_{L^{1}}.
  \end{equation*}
  For the remaining piece, integration by parts gives
  \begin{equation*}
    i A \int_{\epsilon}^{\infty}=
    -e^{i \epsilon A}\chi(\epsilon)\epsilon^{-1/2}-
    \int_{\epsilon}^{\infty}e^{i \lambda A}
    \partial_{\lambda}(\chi \lambda^{-1/2})
  \end{equation*}
  so that
  \begin{equation*}
    \left|A \int_{\epsilon}^{\infty}\right|\le
    \|\chi\|_{L^{\infty}}\epsilon^{-1/2}+
    \int_{\epsilon}^{\infty}|\chi'| \lambda^{-1/2}+
    \frac12\int_{\epsilon}^{\infty}|\chi| \lambda^{-3/2}
    \le
    3\|\chi'\|_{L^{1}}\epsilon^{-1/2}.
  \end{equation*}
  Thus the complete integral satisfies
  \begin{equation*}
    \left|\int_{0}^{\infty}\right|\le
    \|\chi'\|_{L^{1}}
    \left(
    2 \epsilon^{1/2}+\frac{3}{\epsilon^{1/2}A}
    \right)
  \end{equation*}
  and chosing $\epsilon=3/(2A)$ we obtain \eqref{eq:phase1}.
\end{proof}

\begin{lemma}\label{lem:phase2}
  For all $\chi\in C^{2}_{c}(\mathbb{R})$ and $A\neq0$,
  \begin{equation}\label{eq:phase2}
    \left|
    \int_{0}^{\infty}e^{i \lambda A}\chi(\lambda)
    \lambda^{1/2}d \lambda
    \right|\le
    5\|\chi''/2+\lambda \chi'\|_{L^{1}}\cdot |A|^{-3/2}.
  \end{equation}
\end{lemma}

\begin{proof}
  It is sufficient to apply the previous Lemma to the identity
  \begin{equation*}
    iA\int_{0}^{\infty}=
      -\int_{0}^{\infty}e^{i \lambda A}
      \left(\lambda
      \chi'+\chi/2
      \right)\lambda^{-1/2}d \lambda.
  \end{equation*}
\end{proof}


\bibliographystyle{plain}
\bibliography{/Users/piero/Documents/Biblioteca/-bib/bibliodatabase.bib}





\end{document}